\documentclass[final]{siamart1116}

\usepackage{lipsum}
\usepackage{amsfonts,amsmath}
\usepackage{graphicx}
\usepackage{epstopdf}
\usepackage{algorithmic}
\usepackage[color]{showkeys}
\definecolor{refkey}{gray}{.75}

\def\p{\partial}

\def\bs{\boldsymbol}

\def\nb{\mathbf{n}}
\def\ub{\bs{u}}

\def\C{\mathcal{C}}
\def\D{\mathcal{D}}
\def\M{\mathcal{M}}
\def\lpm{{L^p(\M)}}
\def\linfm{L^\infty{\M}}
\def\l2m{L^2(\M)}
\def\w1pm{W^{1,p}(\M)}
\def\h10m{H^1_0(\M)}

\ifpdf
  \DeclareGraphicsExtensions{.eps,.pdf,.png,.jpg}
\else
  \DeclareGraphicsExtensions{.eps}
\fi

\numberwithin{theorem}{section}

\newcommand{\TheTitle}{The barotropic quasi-geostrophic equation under
  a free surface} 
\newcommand{\TheAuthors}{Qingshan Chen}

\headers{\TheTitle}{\TheAuthors}

\title{{\TheTitle}\thanks{Submitted to the editors DATE.
\funding{This work was in part supported by the Simons Foundation
  contract no.~\#319070.}}} 

\author{
  Qingshan Chen\thanks{Department of Mathematical Sciences, Clemson
    University, Clemson, SC 29631.
    (\email{qsc@clemson.edu}, \url{http://discreteinfo.net}).}
}

\usepackage{amsopn}

\ifpdf
\hypersetup{
  pdftitle={\TheTitle},
  pdfauthor={\TheAuthors}
}
\fi

\begin{document}
\maketitle

\begin{abstract}
The inviscid barotropic quasi-geostrophic equation with a free surface
is considered. The free surface mandates a non-standard boundary
condition. The global existence existence and uniqueness of a weak
solution is established, thanks to the uniform in time bounds on the
potential vorticity. The solution is also shown to satisfy the initial
and boundary conditions in the classical sense. 
\end{abstract}

\begin{keywords}
Partial differential equations, inviscid model, geophysical fluid
dynamics, global well-posedness
\end{keywords}

\begin{AMS}
35Q35, 35Q86
\end{AMS}

\section{Introduction}

The inviscid barotropic quasi-geostrophic equation (QG) for large-scale
geophysical flows takes the form of a  scalar transport
equation,
\begin{equation}
  \label{eq:1}
  \dfrac{\p}{\p t} q + \ub \cdot\nabla q = f,\qquad \M.
\end{equation}
Here,
\begin{equation}
  \label{eq:2}
  q = \nabla\times\ub + y -\psi
\end{equation}
is the QG potential vorticity (QGPV), and the transport velocity $\ub$
is given by
\begin{equation}
  \label{eq:3}
  \ub = \nabla^\perp\psi.
\end{equation}
The streamfunction $\psi$ physically represents small perturbations to
the surface height. The equation is posed on a simply-connected
two-dimensional domain $\M$. 
Within the expression \eqref{eq:2} for the QGPV $q$, $\nabla\times\ub$ is
the relative vorticity of the velocity field, $y$ the so-called beta
term that arises thanks to the differential effect of the earth
rotation along the meridional direction, and $\psi$ the surface
deformation.


For the simplicity of presentation, all the variables in
\eqref{eq:1}-\eqref{eq:3} have been non-dimensionalized, and the
dimensionless coefficients have been rounded to the constant 1,
mandating that each term be of equal significance to the dynamics. In
reality, of course, the situation is much more complex. We point out
that, when the horizontal length scale is much smaller than the Rossby
deformation radius, the fluctuations of the top surface are small and
their impact 
on the vorticity dynamics is negligible, i.e.~the classical ``rigid
lid'' assumption holds. 
However, in
non-homogeneous fluids, since the horizontal length scale of the flow
are close in scale to the Rossby deformation radii of the interior
layer interfaces, the deformations of the interior layer interface are
greater, and so are their impact on the vorticity dynamics. For this
reason, we want to study the well-posedness of the QG equation when
the surface deformation is included. The current work can be
considered a preparation for future efforts on more complex systems. 

Equation \eqref{eq:1} implies that, in the absence of an
external forcing, i.e.~$f=0$, the QGPV is conserved along the fluid
paths. There is, of course, an analogue between the QG equation and
the two-dimensional incompressible Euler equation, where the vorticity
(called the relative vorticity in the above) is conserved along the
fluid paths. The main difference between them is the appearance of the
stream-function $\psi$ in the QGPV, which reflects the fact that the
QG actually describes a three-dimensional body of fluids, whose top
surface is free to deform, while the two-dimensional Euler is strictly
for the planar fluids. But the motions of the three-dimensional body
of fluid are assumed to be uniform across the fluid depth, and
therefore the velocity field $\ub$ is two-dimensional. The third
dimension only manifests itself through the varying fluid depth. A
review of the derivation of the QG (see e.g.~\cite{Pedlosky1987-gk,
  Vallis2006-jj}) reveals that the surface fluctuations affects the
vorticity dynamics thanks to the mild compressibility of the
full two-dimensional velocity field; the $\ub$ in
\eqref{eq:1}, which is incompressible, is the leading component in the asymptotic
expansion of the full velocity
field. Hence, the surface fluctuation $\psi$ 
in \eqref{eq:2} is a manifestation of the additional variability in
the vertical dimension and the mild compressibility of the
two-dimensional velocity field. 

In a broader context, many geophysical models are derived based on the
characteristics of large-scale geophysical flows, such as the small
vertical to horizontal aspect ratio and the strong earth rotation. In 
most cases, the models demonstrate richer and more complex dynamics
than the strictly two-dimensional models, but remain much simpler than
the full three-dimensional fluid models, i.e.~the three-dimensional
Navier-Stokes equations (NSE) or the Euler equations. Mathematically
speaking, the two-dimensional fluid models, i.e.~the two-dimensional
Navier-Stokes or Euler equations, are well understood. The existence,
uniqueness, and regularity of a global solution to these equations are
known (\cite{Yudovich1963-bj, Kato1967-cj, Bardos1972-ar}). But the
same cannot be said about the three-dimensional NS or Euler equations;
for a review of the limited results on these equations, see
\cite{Temam2001-th, Bardos2007-cf}. Situated between the purely
two-dimensional fluid models and the full three-dimensional models,
large-scale geophysical flow models, such as the primitive equations
(PEs), the shallow water equations, and the quasi-geostrophic
equations, can offer valuable insights into the complex dynamics of
fluid flows, and help bridge the knowledge gap between two-dimensional and
three-dimensional fluid models. 

Partly for the reason mentioned above, and partly for the practical
interests in the evolution of large-scale geophysical flows,
geophysical fluid models have been the subject of intense effort in
the mathematical community for the past few decades. Lions, Temam, and
Wang (\cite{Lions1995-zh, Lions1993-io, Lions1993-sx, Lions1992-ur,
  Lions1992-zo}) offered the first systematic and rigorous treatment
of the three-dimensional viscous primitive equations. Their results
were followed and improved by many subsequent works; for a review of
these progresses, see the review article \cite{Petcu2009-rs}. In
particular, Cao and Titi (\cite{Cao2007-ov}) and Kobelkov
(\cite{Kobelkov2007-vy, Kobelkov2006-lw}) independently established
the global existence and uniqueness of a strong solution to the
three-dimensional viscous PEs under the rigid-lid assumption. 

On the side of quasi-geostrophic equations, several authors have
studied the three-dimensional QG equation under idealized settings, in
the unbounded half space, or a rectangular box. An early work
is by Dutton (\cite{Dutton1974-oo, Dutton1976-kh}), who considered the
three-dimensional QG model in a rectangular box with periodic boundary
conditions on the sides, and homogeneous Neumann boundary conditions
on the top and bottom. The uniqueness of a classical solution, if it
exists, and the global existence of a generalized solution were
established. Bourgeois and Beale (\cite{Bourgeois1994-tv}) studied the
equation in a similar setting, and the existence of a global strong
solution was proved. Desjardins and Greneier
(\cite{Desjardins1998-hb}) also considered the equation in a similar
setting, but included in their model the Ekman pumping effect which
effectively add diffusion to the flow. The existence of a global weak
solution is given. Puel and Vasseur (\cite{Puel2015-mw}) considered
the inviscid QG in the upper half space, with the non-penetration
boundary condition at the bottom of the fluid. The global existence of
a weak solution was proven.  In these works, the issue of uniqueness
of the solutions was left open. In a recent work, Novack and Vasseur
(\cite{Novack2016-op}) considered the three-dimensional QG in the same
spatial setting as in \cite{Puel2015-mw}, but with an added diffusion
term in the boundary at $z=0$ due to the Ekman pumping effect. The
existence and uniqueness of a global strong solution is proven. 

Another related model is the surface QG equation (SQG). The SQG is in
fact a generalization of the top surface boundary condition for the
three-dimensional QG. The curl form of the SQG resembles that of the
three-dimensional Euler equations. For this reason, the SQG has been
intensely studied in the past twenty years or so
(\cite{Constantin1994-me, Constantin1994-gm, Held1995-aa,
  Constantin1998-az, Constantin1999-yi,Li2008-hv,  
  Caffarelli2010-iq, Kiselev2010-ka,
  Friedlander2011-wb, Chae2012-sv, Dabkowski2012-zu,    Miao2012-an, 
  Lazar2013-wf, Fefferman2015-ww, 
   Cheskidov2015-rj, Kukavica2015-wh,
   Castro2016-gw, Buckmaster2016-gw}). 

The current work studies the inviscid barotropic QG equation under a
free surface and on a general bounded domain. Without the free
surface, the barotropic QG is 
mathematically equivalent to the two-dimensional incompressible Euler
equation, whose well-posedness has been established by various authors
(\cite{Yudovich1963-bj, Kato1967-cj}). The introduction of the the
free surface not only changes the relation between the potential
vorticity $q$ and the streamfunction $\psi$, but also mandates a new
and slightly more complicated type of boundary condition for the stream
function. 
For the two-dimensional Euler equation, the homogeneous Dirichlet
boundary conditions suffice for the streamfunction. But this is no
longer true when the top surface is left free.
The non-penetration boundary condition on the flow mandates that
the streamfunction be constant along the boundary. 
Physically, the constant boundary value of the streamfunction should
be left free to accommodate the free deformation of the top
surface. Mathematically, that constant boundary value cannot be
arbitrarily set without altering the shape of the solution, unlike in
the case of the two-dimensional Euler equations. Thus, additional
constraints have to be introduced to determine the value of the
streamfunction on the boundary.
In this work, we determine the constant by
enforcing the mass conservation condition. 

The constant but non-zero boundary condition gives rise to several
technical difficulties that were absent in the case of the
two-dimensional Euler equation. The main contribution of this work is
to address these difficulties and establish the well-posedness of the
barotropic QG equation with a free surface. The proof follows the
approach that was originally laid down by Yudovich
(\cite{Yudovich1963-bj}). However, for the construction of the flow
map, the approach from Marchioro (\cite{Marchioro1994-yt}) is
adopted. The simpler 
approach of Kato (\cite{Kato1967-cj}) does not apply because the solution of the
current problem is not sufficiently smooth.

\section{The initial and boundary conditions}
It will become clear later in the analysis that the streamfunction
$\psi$ is a key quantity in the QG dynamics. In fact, the QG equation
can be expressed entirely in this quantity,
\begin{equation}
  \label{eq:4}
  \dfrac{\p}{\p t}(\Delta\psi +y -\psi) +\nabla^\perp\psi\cdot
  \nabla(\Delta\psi + y - \psi) = f,\quad \M.
\end{equation}

Since the model is inviscid, it is natural to impose the no-flux boundary
condition on the domain boundary $\p\M$, 
\begin{equation}
  \label{eq:5}
  \ub\cdot\nb = -\dfrac{\p\psi}{\p\tau} = 0,\qquad \p\M,
\end{equation}
where $\nb$ and $\tau$ stand for the outer normal and positively
oriented tangential vectors, respectively, on the boundary. The
condition \eqref{eq:5} is equivalent to the requirement that $\psi$ be
constant along the boundary, i.e.~for some quantity $l$ that depends
on $t$ only,
\begin{equation}
  \label{eq:6}
  \psi(x, t) = l(t),\qquad\forall\, x\in\p\M.
\end{equation}
The boundary value of $\psi$ has been left free to accommodate the
free movement of the top surface. In order to determine the value $l$
for each $t$,
we require that the fluctuation of the
surface does not affect the overall volume of the fluid, that
is,
\begin{equation}
  \label{eq:7}
  \int_\M \psi dx = 0.
\end{equation}
This is equivalent to the condition that mass is conserved.

The initial condition is specified on the streamfunction as well,
\begin{equation}
  \label{eq:8}
  \psi(x,0) = \psi_0(x),\qquad \forall\,x\in\M.
\end{equation}

Thus, equations \eqref{eq:1}--\eqref{eq:3} and \eqref{eq:6}--\eqref{eq:8}
(or, equivalently, \eqref{eq:4} and \eqref{eq:6}--\eqref{eq:8})
constitute the complete initial and boundary value problem of the
barotropic QG equation.

In the QG equation, the right-hand side forcing $f$ is typically the
curl of a vector field $\bs{F}$, representing, e.g., the wind in the
physical world. Hence, we assume that $f$ is derived from a given
vector field $\bs{F}$ via
\begin{equation}
  \label{eq:forcing}
  f = \nabla\times\bs{F}.
\end{equation}

\section{An non-standard elliptic boundary value
  problem}\label{sec:bvp} 
For regularity, we assume that the boundary of the domain, $\p \M$, is
at least 
$C^2$ smooth. 

Once the QGPV $q$ is known, the streamfunction $\psi$ can be
determined from an non-standard elliptic boundary value problem,
\begin{subequations}
  \label{eq:9}
  \begin{align}
    \Delta \psi - \psi &= q - y, & \M&,\label{eq:9a}\\
    \psi &= l, & \p\M&,\label{eq:9b}\\
    \int_\M\psi dx &= 0. & &\label{eq:9c}
  \end{align}
\end{subequations}
We proceed by decomposition. This technique can be
applied in more complex situations with holes inside the domain (see
\cite{Kato1967-cj}). We let $\psi_1$ and $\psi_2$ be solutions of the
following elliptic BVPs, respectively,
\begin{subequations}
  \label{eq:10}
  \begin{align}
    \Delta \psi_1 - \psi_1 &= q - y, & \M&,\label{eq:10a}\\
    \psi_1 &= 0, & \p\M&,\label{eq:10b}
  \end{align}
\end{subequations}
and
\begin{subequations}
  \label{eq:11}
  \begin{align}
    \Delta \psi_2 - \psi_2 &= 0, & \M&,\label{eq:11a}\\
    \psi_2 &= 1, & \p\M&.\label{eq:11b}
  \end{align}
\end{subequations}
By the standard elliptic PDE theories, both BVPs \eqref{eq:10} and
\eqref{eq:11} are well-posed under proper assumptions on the forcing
on the right-hand side of \eqref{eq:10a} and on the domain $\M$.
The solution to the original BVP \eqref{eq:9} can be expressed in terms
of $\psi_1$ and $\psi_2$,
\begin{equation}
  \label{eq:12}
  \psi = \psi_1 + l\psi_2.
\end{equation}
The unknown constant $l$ can be determined using the mass conservation
constraint
\eqref{eq:9c}
\begin{equation*}
  \int_\M \psi dx = \int_\M \psi_1 dx + l\int_\M \psi_2 dx =
  0,
\end{equation*}
which leads to 
\begin{equation}\label{eq:13}
l = -\dfrac{\displaystyle\int_\M \psi_1 dx}{\displaystyle\int_\M \psi_2 dx}. 
\end{equation}
We point out that the expression \eqref{eq:13} for $l$ is valid
because, as a consequence of the maximum principle, $\psi_2$ 
is positive in the interior of the domain, and the integral of $\psi_2$
in the denominator of \eqref{eq:13} is strictly positive.

The elliptic PDE \eqref{eq:9a} is called the Helmholtz equation. The
associated differential operator, $\Delta - I$, has a fundamental
solution that is all regular except for a logarithmic singularity 
(\cite{Courant1989-gn}), just like the Laplacian operator
$\Delta$. Thus, provided that the boundary $\p\M$ is sufficiently
smooth, the Green's function $G(x,y)$ for the elliptic BVP
\eqref{eq:10} exists, and is smooth except for a logarithmic
singularity. Specifically, $G(x,y)$ has the following forms and
estimates, 
for $\forall$ $x,\,y\in\M$,
\begin{subequations}
  \label{eq:13a}
  \begin{align}
    &G(x,y) = a(x,y) \ln |x-y| + b(x,y), \label{eq:13ac}\\ 
  &\dfrac{\p G}{\p x_i} (x,y) = \dfrac{c(x,y)}{|x-y|} +
                               d(x,y),\label{eq:13ad}\\ 
    &|G(x,y)| \leq
    C(\M)\left(1+\left|\ln|x-y|\right|\right),\label{eq:13aa}\\
    &\left|\dfrac{\p G}{\p x_i}(x,y)\right| \leq C(\M)(1+|x-y|^{-1}),\qquad 
    i=1,\,2. \label{eq:13ab} 
  \end{align}
\end{subequations}
In the above, $a(x,y)$, $b(x,y)$, $c(x,y)$, and $d(x,y)$ are functions
that are regular over the entire domain $\M$, and whose maximum values
depend on $\M$ only.

Using the Green's function $G(x,y)$, the solution $\psi_1$ of
\eqref{eq:10} can be written as  
\begin{equation}
  \label{eq:13b}
  \psi_1(x) = \int_\M G(x,y)(q(y) - y_2)dy. 
\end{equation}
The solution $\psi_2$ of \eqref{eq:11} can be written as 
\begin{equation}
  \label{eq:13c}
  \psi_2(x) = 1+ \int_\M G(x,y)dy. 
\end{equation}
 Substituting \eqref{eq:13b} and \eqref{eq:13c} into \eqref{eq:13}, we
 obtain an expression for $l$, as a functional of the QGPV $q$,
 \begin{equation}
   \label{eq:13d}
   l(q) = -\dfrac{\displaystyle\int_\M \int_\M G(x,y)(q(y)-y_2)
     dydx}{\displaystyle|\M| + \int_\M \int_\M G(x,y)dydx}.
 \end{equation}
The solution  $\psi$ to the non-standard BVP \eqref{eq:9} can be
expressed as 
\begin{equation}
  \label{eq:13e}
  \psi(x) = \int_\M G(x,y)(q(y)-y_2) dy + l (1 + \int_\M G(x,y)dy).
\end{equation}

As in the case of the two-dimensional Euler equation, the QGPV is
simply being advected by the velocity field, and its maximum values
are preserved in the absence of an external forcing. 
As noted above on \eqref{eq:13}, the denominator on the right-hand
side of \eqref{eq:13d}  is strictly positive. Thus, from
\eqref{eq:13d}, and making use of \eqref{eq:13aa}, one derives an
estimate on the magnitude of the constant value $l$ of $\psi$ on the
boundary, 
\begin{align*}
  |l| 
   &\leq  C(\M, |q|_\infty)\left(1+\int_\M \int_\M
     \left|\ln|x-y|\right| dy dx\right)\\
   &\leq  C(\M, |q|_\infty)\left(1+\int_\M
     \left(\int_{\substack{|y-x|\ge 1\\y\in\M}} 
     \left|\ln|x-y|\right| dy + \int_{\substack{|y-x|\le 1\\y\in\M}} 
     \left|\ln|x-y|\right| dy\right) dx \right)\\
   &\leq  C(\M, |q|_\infty)\left(1+\int_\M
     \int_{\substack{|y-x|\le 1\\y\in\M}} 
     \left|\ln|x-y|\right| dy dx \right)\\
   &\leq  C(\M, |q|_\infty)\left(1+\int_\M
     \int_{|y-x|\le 1} 
     \left|\ln|x-y|\right| dy dx \right)\\
   &\leq  C(\M, |q|_\infty)\left(1+\int_\M
     \int_0^1 2\pi r 
     \left|\ln r|\right| dr dx \right)\\
  &\leq  C(\M, |q|_\infty)\left(1+2\pi |\M|
     \int_0^1 
     r \left|\ln r\right|dr  \right)\\
  &\leq  C(\M, |q|_\infty)\left(1+\dfrac{2\pi |\M|}{e}
       \right).
\end{align*}
To summarize, we have just shown that the value of $\psi$ on the
boundary is bounded by a constant that depends on the
domain and the maximum norm of the QGPV $q$, that is,
\begin{equation}
  \label{eq:13f}
  |l| \leq C(\M, |q|_\infty).
\end{equation}

Below, we shall formally state the regularity results for the elliptic
boundary value problem \eqref{eq:9}. But, in order to do so, we need
to first give the precise definitions of some relevant function
spaces. 

We denote by $Q_T$ the spatial-temporal domain,
\begin{equation*}
  Q_T = \M\times (0,T).
\end{equation*}
We denote by $L^\infty(\M)$, or $L^\infty(Q_T)$ when time is also
involved, the space of functions that are essentially bounded. 
We denote by $C^{0,\gamma}(\M)$, with $\gamma>0$, the space of
H\"older-continuous functions on $\M$, and similarly,
$C^{0,\gamma}(Q_T)$ on $Q_T$. $C^{0,\gamma}(\M)$ and
$C^{0,\gamma}(Q_T)$ are both Banach spaces under the usual H\"older
norms. 

We denote by $V$ the space of solutions to the elliptic boundary value problem \eqref{eq:9} with $q\in L^\infty(\M)$, i.e.,
\begin{equation*}
  V := \left\{ \psi\,|\, \psi \textrm{ solves } \eqref{eq:9} \textrm{
      for some }  q\in L^\infty(\M)\right\}.
\end{equation*}
The space $V$ is equipped with the norm
\begin{equation*}
  \| \psi\| _V := \| \Delta\psi -\psi\|_{L^\infty(\M)}.
\end{equation*}
By the continuity of the inverse elliptic operator $(\Delta -I)^{-1}$,
$V$ is a Banach space.

In the analysis, we will also encounter functions that are
differentiable with continuous first derivatives. The space of these
functions will be denoted as $C^1(\M)$, equipped with the usual $C^1$
norm. 

When time is involved, we use $L^\infty(0,T; V)$ to designate the
space of functions that are \emph{essentially} bounded with respect to
the $\| \cdot\|_V$ norm, and $L^\infty(0,T; C^1(\M))$ for
functions that are \emph{essentially} bounded under the
$\|\cdot\|_{C^1(\M)}$ norm. 

We can now formally state the regularity result for the elliptic
boundary value problem \eqref{eq:9}.
\begin{lemma}\label{lem:elliptic-reg}
  Let the boundary $\p\M\in C^2$, $q\in L^\infty(\M)$. Then the
  solution $\psi$ belongs to $W^{2,p}(\M)$ for any $p>1$, and
  the derivatives enjoy the following estimates,
  \begin{equation}
    \label{eq:14}
    \| D^2\psi\|_{L^p(\M)} \leq C p \| q-y \|_{L^\infty(\M)},
  \end{equation}
where $D$ denotes the first-order differential operator, and the
constant $C$ depends on $\M$ only, and not on $p$ or the potential
vorticity $q$. The first derivatives of the function $\psi$ satisfy
the H\"older condition with any $0<\gamma < 1$,
\begin{equation}
  \label{eq:15}
  \| D\psi\|_{C^{0,\gamma}(\M)} \leq \dfrac{C}{1-\gamma}
  \| q-y\|_{L^\infty(\M)},
\end{equation}
and the quasi-Lipschitz condition,
\begin{equation}
  \label{eq:16}
  |D\psi(\xi) - D\psi(\eta) | \leq C\chi(\delta) \| q-y\|_{L^\infty(\M)},
\end{equation}
where $\delta = |\xi - \eta|$, and 
\begin{equation*}
  \chi(\delta) = \left\{
    \begin{aligned}
      &(1-\ln\delta)\delta & &\textrm{if } \delta < 1,\\
      &1 & &\textrm{if } \delta \ge 1.
    \end{aligned}\right.
\end{equation*}
\end{lemma}
This result is similar to the one given in \cite{Yudovich1963-bj} for
the Euler equation. What is new here is the presence of a free surface
and its constant value on the boundary.

\begin{proof}[Proof of Lemma \ref{lem:elliptic-reg}]
The assertion \eqref{eq:14} is part of the classical $L^p$ regularity
theory for the elliptic BVPs (\cite{Gilbarg1983-pq}). We now verify
that $\psi\in C^1(\M)$ and the H\"older condition \eqref{eq:15} and
the quasi-Lipschitz condition \eqref{eq:16} hold for $\psi$. We
formally differentiate \eqref{eq:13e} with respect to $x_i$,
\begin{equation}
  \label{eq:14a}
    \dfrac{\p\psi}{\p x_i}(x) = \int_\M \dfrac{\p G}{\p
      x_i}(x,y)(q(y)-y_2) dy + l\int_\M \dfrac{\p G}{\p x_i}(x,y)dy.
\end{equation}
We call the right-hand side $u(x)$, and we need to show that $u(x)$ is
well-defined. Using the fact that $q$ is essentially bounded, we find
that 
\begin{align*}
  |u(x)| &\leq \int_\M \left|\dfrac{\p G}{\p x_i}(x,y) \right|\cdot
           |q(y) - y_2| dy + l \int_\M \left|\dfrac{\p G}{\p
           x_i}(x,y)\right| dy \\
  &\leq \left(|q|_\infty + |y_2|_\infty + |l|\right)\int_\M \left|\dfrac{\p G}{\p
           x_i}(x,y)\right| dy.
\end{align*}
Using the estimate \eqref{eq:13ab} in the above, we proceed with the estimates,
\begin{align*}
 |u(x)|  &\leq \left(|q|_\infty + |y_2|_\infty + |l|\right)\int_\M
    \dfrac{1}{|y-x|} dy\\
  & \leq \left(|q|_\infty + |y_2|_\infty + |l|\right) \left\{
    \int_{\substack{|y-x|\ge 1\\y\in \M}} \dfrac{1}{|y-x|} + \int_{|y-x|\leq 1}
  \dfrac{1}{|y-x} dy\right\}\\
  & \leq \left(|q|_\infty + |y_2|_\infty + |l|\right) \left\{
    |\M|\cdot 1 + \int_0^1\int_{|y-x|=r}
  \dfrac{1}{r} ds dr \right\}\\
  & \leq \left(|q|_\infty + |y_2|_\infty + |l|\right)\cdot (|\M| + 2\pi).
\end{align*}
Combining this estimate with the estimate \eqref{eq:13f}, we reach
\begin{equation}
  \label{eq:14b}
    |u(x)| \leq C(\M, |q|_\infty).
\end{equation}
Hence, the right-hand side of \eqref{eq:14a} is well-defined, and the
relation \eqref{eq:14a} holds. 

Next, we show that the first derivative of $\psi$ is quasi-Lipschitz
continuous and satisfies the estimates \eqref{eq:16}. We let $\xi$
and $\eta$ be two arbitrary points in $\M$. Then, using the relation
\eqref{eq:14a}, we derive that 
\begin{equation}
  \label{eq:14c}
  \left| \dfrac{\p\psi}{\p x_i}(\xi) - \dfrac{\p\psi}{\p x_i}(\eta)
  \right| \leq \left(|q-y_2|_\infty + |l|\right) \int_\M
  \left|\dfrac{\p G}{\p x_i}(\xi, y) - \dfrac{\p G}{\p
      x_i}(\eta, y)\right| dy.
\end{equation}
Substituting the form \eqref{eq:13ad} into \eqref{eq:14c}, and using the
estimate \eqref{eq:13f} on $l$, one derives that 
\begin{equation}
  \label{eq:14e}
  \left| \dfrac{\p\psi}{\p x_i}(\xi) - \dfrac{\p\psi}{\p x_i}(\eta)
  \right| \leq C(\M, |q|_\infty) \left(|\xi - \eta| + \int_\M
    \left|\dfrac{1}{|y-\xi|} - \dfrac{1}{|y-\eta|}\right| dy\right).
\end{equation}
Using the triangular inequality, one finds that, in a space of general
dimension $n$, 
\begin{equation}
  \label{eq:14f}
  \left| \dfrac{1}{|y-\xi|^{n-1}} - \dfrac{1}{|y-\eta|^{n-1}}\right|
  \leq \dfrac{|\xi-\eta|}{|y-\xi^\ast(y)|^n},
\end{equation}
where $\xi^\ast(y)$ is a point between $\xi$ and $\eta$, and depends on
$y$. We let $R>0$ be such that $B(\xi, R)$, the ball centered
at $\xi$ with radius $R$, covers the entire domain $\M$. 
We then decompose the domain into two parts, one within the small ball
$B(\xi,2\delta)$, and the other within the annulus $2\delta\leq
|y-\xi| \leq R$, where $\delta = |\xi-\|$. We estimate the integral of
the left-hand side of \eqref{eq:14f} in these two sub-domains
separately. 
\begin{align*}
  &\int_\M
    \left|\dfrac{1}{|y-\xi|^{n-1}} - \dfrac{1}{|y-\eta|^{n-1}}\right| dy \\
\leq& \int_{|y-\xi|\leq 2\delta}
    \left|\dfrac{1}{|y-\xi|^{n-1}} - \dfrac{1}{|y-\eta|^{n-1}}\right| dy +
  \int_{2\delta < |y-\xi| < R}
      \left|\dfrac{|\xi-\eta|}{|y-\xi|^{n-1}} - \dfrac{1}{|y-\eta|^{n-1}}\right| dy\\
 \leq& \int_{|y-\xi|\leq 2\delta}
     \dfrac{1}{|y-\xi|^{n-1}}dy +\int_{|y-\eta|\leq 3\delta} \dfrac{1}{|y-\eta|^{n-1}} dy +
 \int_{2\delta < |y-\xi| < R}
       |\dfrac{|\xi-\eta|}{|y-\xi|^n}\dfrac{|y-\xi|^n}{|y-\xi^*|^n} dy\\
\leq& \int_0^{2\delta} \int_{|y-\xi|=r}\dfrac{1}{r^{n-1}} ds dr +
      \int_0^{3\delta} \int_{|y-\eta|=r}\dfrac{1}{r^{n-1}} ds dr + 
      \int_{2\delta}^{R} \int_{|y-\xi|=r} 2^n \dfrac{|\xi-\eta|}{r^{n}} ds dr\\
\leq& \int_0^{2\delta} \omega_n r^{-(n-1)} r^{n-1} dr +
      \int_0^{3\delta} \omega_n r^{-(n-1)} r^{n-1} dr + 
      \int_{2\delta}^{R} |\xi-\eta| r^{-n} \omega_n r^{n-1} dr\\
\leq& \left(5 + 2^n(\ln R - \ln 2) - 2^n \ln \delta\right) \omega_n \delta.
\end{align*}
Using this estimate in \eqref{eq:14e}, we obtain
 \begin{equation}
  \label{eq:14g}
     \left| \dfrac{\p\psi}{\p x_i}(\xi) - \dfrac{\p\psi}{\p x_i}(\eta)
   \right| \leq C(\M, |q|_\infty) (1-\ln \delta)\delta.
 \end{equation}
Thus, the claim of the quasi-Lipschitz continuity is proven. 

The H\"older continuity is a consequence of the quasi-Lipschitz
continuity. Indeed, for any $0 < \lambda < 1$, 
\begin{equation*}
  \dfrac{\left| \dfrac{\p\psi}{\p x_i}(\xi) - \dfrac{\p\psi}{\p
        x_i}(\eta)\right|}{|\xi - \eta|^\lambda} \leq C(\M,
  |q|_\infty) (1-\ln\delta) \delta^{1-\lambda}.
\end{equation*}
The $0<\delta<1$, the expression $|\ln\delta|\cdot\delta^{1-\lambda}$
has a maximum value of $e^{-1}/(1-\lambda)$. Hence, we have that 
\begin{equation*}
  \dfrac{\left| \dfrac{\p\psi}{\p x_i}(\xi) - \dfrac{\p\psi}{\p
        x_i}(\eta)\right|}{|\xi - \eta|^\lambda} \leq C(\M,
  |q|_\infty) \dfrac{1}{1-\lambda}.
\end{equation*}
The lemma is proven.
\end{proof}

In the sequel, we will need the following regularity result, which can be easily
derived from the classical $L^p$ theory for elliptic equations with
Dirichlet boundary conditions (\cite{Gilbarg1983-pq}). 
\begin{lemma}\label{lem:elliptic-reg1}
  Let $g\in L^p(\M)$ with $p>1$, and let $\psi$ be a solution of 
\begin{subequations}
  \label{eq:16a}
  \begin{align}
    \Delta \psi - \psi &= \dfrac{\p g}{\p x_i}, & \M&,\label{eq:16aa}\\
    \psi &= l, & \p\M&,\label{eq:16ab}\\
    \int_\M\psi dx &= 0. & &\label{eq:16ac}
  \end{align}
\end{subequations}
Then, $\psi$ has one generalized derivative, and
\begin{equation}
  \label{eq:16b}
  \| \psi\|_{W^{1,p}(\M)} \leq Cp \|g\|_{L^p(\M)}
\end{equation}
\end{lemma}

\section{Weak formulation and the uniqueness}
We assume that $\psi$ is a classical solution of \eqref{eq:4} subjecting
to the constraints \eqref{eq:6}--\eqref{eq:8}. We let $\varphi\in
C^\infty(Q_T)$ with $\varphi|_{\p\M} = \varphi|_{t=T} = 0$. We
multiply \eqref{eq:4} with $\varphi$ and integrate by parts to obtain
\begin{multline}
  \label{eq:17}
  -\int_\M (\Delta\psi_0 - \psi_0)\varphi(x,0) dx - \int_0^T\int_\M
  (\Delta\psi - \psi)\dfrac{\p\varphi}{\p t}dx dt \\
   - \int_0^T\int_\M
  (\Delta\psi + y - \psi)\nabla^\perp\psi\cdot\nabla\varphi dxdt =
  \int_0^T\int_\M f\varphi dxdt.
\end{multline}
Thus, every classical solution of the barotropic QG equation also
solves the integral equation \eqref{eq:17}, but the converse is not
true, 
for the QGPV $q = \Delta\psi + y -\psi$ may not be differentiable
either in space $x$ or in time $t$. Solutions of \eqref{eq:17} are
called weak solutions of the barotropic QG. 

We establish the well-posedness of the barotropic QG  \eqref{eq:4},
\eqref{eq:6}--\eqref{eq:8} by working with its weak formulation first,
whose precise statement is given here.\\
\noindent{\bfseries Statement of the problem:}\\
\begin{equation}\label{problem}
  \begin{aligned}
  &\textrm{Let } \psi_0\in
  V. \textrm{ Find } \psi\in  L^\infty(0,T; V) \textrm{ such that  
  \eqref{eq:17} holds for every } \\
  &\varphi\in C^\infty(Q_T) \textrm{
  with }
 \varphi|_{\p\M} = 0=\varphi|_{t=T} = 0. 
  \end{aligned}
\end{equation}

We choose $\varphi(x,t) = g(t)\gamma(x)$ in \eqref{eq:17} with
$g\in C^\infty([0,T])$, $g(T) = 0$, and  $\gamma\in C^\infty_c(\M)$.
Substituting this $\varphi$ into \eqref{eq:17}, we have
\begin{multline}
  \label{eq:18}
 - g(0) \int_\M (\Delta\psi_0 - \psi_0)\gamma dx - \int_0^Tg'(t)\int_\M
  (\Delta\psi - \psi)\gamma(x)dx dt \\
   - \int_0^Tg(t)\int_\M
  (\Delta\psi + y - \psi)\nabla^\perp\psi\cdot\nabla\gamma dxdt =
  \int_0^Tg(t)\int_\M f\gamma dxdt.
\end{multline}
If we take $g(0) = 0$ as well, then \eqref{eq:18} becomes
\begin{multline}
  \label{eq:19}
  - \int_0^Tg'(t)\int_\M
  (\Delta\psi - \psi)\gamma(x)dx dt    = \\
   \int_0^Tg(t)\int_\M
  \left((\Delta\psi + y - \psi)\nabla^\perp\psi\cdot\nabla\gamma +
    f\gamma\right) dxdt. 
\end{multline}
This shows that 
\begin{multline}
  \label{eq:20}
  \dfrac{d}{d t}\int_\M
  (\Delta\psi - \psi)\gamma(x)dx dt    = \\ \int_\M
  \left((\Delta\psi + y - \psi)\nabla^\perp\psi\cdot\nabla\gamma + f\gamma
  \right) dx\qquad \textrm{in }\D'(0,T).  
\end{multline}
Thanks to the fact that $C^\infty_c(\M)$ is dense in $H^1_0(\M)$, the
above also holds for every $\varphi\in H^1_0(\M)$. Thus, we conclude
that  $\Delta\psi - \psi$ is weakly continuous in time in the following
sense, 
\begin{equation*}
 {\int_\M (\Delta\psi -\psi)\gamma dx \textrm{ is continuous in time
    for every } \gamma\in H^1_0(\M).}
\end{equation*}

Integrating by parts in \eqref{eq:18}, we find 
\begin{multline}
  \label{eq:21}
 - g(0) \int_\M (\nabla\psi_0\cdot\nabla\gamma + \psi_0\gamma) dx +
 \int_0^Tg'(t)\int_\M 
  (\nabla\psi\cdot\nabla\gamma + \psi\gamma)dx dt = \\
    \int_0^Tg(t)\int_\M
  \left((\Delta\psi + y - \psi)\nabla^\perp\psi\cdot\nabla\gamma +
    f\gamma\right) dxdt. 
\end{multline}
Again, taking $g(0) = 0$ yields
\begin{equation}
  \label{eq:22}
 \int_0^Tg'(t)\int_\M 
  (\nabla\psi\cdot\nabla\gamma + \psi\gamma)dx dt =
    \int_0^Tg(t)\int_\M
  \left( (\Delta\psi + y - \psi)\nabla^\perp\psi\cdot\nabla\gamma +
    f\gamma\right) dxdt. 
\end{equation}
Since $C^\infty_c(\M)$ is dense in the space $H^1_0(\M)$ under the
usual $H^1$-norm, the above holds for every $\gamma\in
H^1_0(\M)$. Thus, 
\begin{equation}
  \label{eq:23}
  \dfrac{d}{dt}\int_\M 
  (\nabla\psi\cdot\nabla\gamma + \psi\gamma)dx =
    -\int_\M \left(
  (\Delta\psi + y - \psi)\nabla^\perp\psi\cdot\nabla\gamma + f\gamma
  \right) dx\quad\textrm{ in } \D'(0,T).
\end{equation}
This implies that $\psi$ is weakly continuous in time for the
$H^1$-norm, 
\begin{equation*}
{\int_M(\nabla\psi\cdot\nabla\gamma + \psi\gamma)dx \textrm{ is
    continuous in time for every } \gamma\in H^1_0(\M). }
\end{equation*}

To investigate the initial value of $\psi$, we take
$g\in\C^\infty([0,T])$ with $g(0) \ne 0$ and $g(T) = 0$. We multiply
\eqref{eq:20} by $g(t)$ and integrate by parts in $t$ to obtain
\begin{multline}
  \label{eq:24}
  g(0) \int_\M (\Delta\psi(x,0) - \psi(x,0))\gamma dx - \int_0^Tg'(t)\int_\M
  (\Delta\psi - \psi)\gamma(x)dx dt \\
   - \int_0^Tg(t)\int_\M
  (\Delta\psi + y - \psi)\nabla^\perp\psi\cdot\nabla\gamma dxdt =
  \int_0^Tg(t)\int_\M f\gamma dxdt.
\end{multline}
Comparing \eqref{eq:24} with \eqref{eq:18}, we find that
\begin{equation}
  \label{eq:25}
  \int_\M\left( (\Delta\psi(x,0) - \psi(x,0)) - (\Delta\psi_0 -
    \psi_0)\right) \gamma dx = 0,\qquad \forall\, \gamma\in C^\infty_c(\M). 
\end{equation}
Since $C^\infty_c(\M)$ is dense in $L^2(\M)$, the above holds for
every $\gamma\in L^2(\M)$. In particular, setting $\gamma$ to the
function in the parentheses, we reach
\begin{equation*}
  \Delta\psi -\psi |_{t=0} = \Delta\psi_0 - \psi_0
  \qquad\textrm{in } L^2(\M).
\end{equation*}

Multiplying \eqref{eq:23} by the same $g(t)$ and integrating by parts
in time, we obtain
\begin{multline}
  \label{eq:26}
 - g(0) \int_\M (\nabla\psi(x,0)\cdot\nabla\gamma + \psi(x,0)\gamma) dx +
 \int_0^Tg'(t)\int_\M 
  (\nabla\psi\cdot\nabla\gamma + \psi\gamma)dx dt \\
    \int_0^Tg(t)\int_\M\left(
  (\Delta\psi + y - \psi)\nabla^\perp\psi\cdot\nabla\gamma + f\gamma
\right) dxdt.
\end{multline}
Comparing this equation with \eqref{eq:21}, we easily see that 
\begin{equation}
  \label{eq:27}
  \int_\M \left( (\nabla\psi(x,0)-\nabla\psi_0)\cdot\nabla\gamma +
    (\psi(x,0) - \psi_0)\gamma \right) dx = 0,\quad
  \forall\,\gamma\in H^1_0(\M).
\end{equation}
From the above, one can infer that the initial condition \eqref{eq:8}
is satisfied in the $H^1$-norm. Indeed, we let 
\begin{equation}
  \label{eq:27a}
  h(x) = \psi(x,0) - \psi_0(x).
\end{equation}
Both $\psi_0$ and $\psi(\cdot,0)$ assume a constant value on the
boundary, and so does $h(x)$. If $h(x)$ vanishes on the boundary, then
we can simply set $\gamma$ to 
$h(x)$, and the conclusion follows. If $h(x)$ does not vanish on the
boundary, then let $\sigma$ be its constant boundary value, and write
\begin{equation}
  \label{eq:27b}
  h(x) = h^\#(x) + \sigma.
\end{equation}
The value $\sigma$ is related to $h^\#$ via
\begin{equation}
  \label{eq:27c}
  \sigma = -\dfrac{1}{|\M|}\int_\M h^\#(x)dx.
\end{equation}
Then $h^\#$ vanishes on the boundary and belongs to $H^1_0(\M)$. 
Substituting \eqref{eq:27b} and \eqref{eq:27c} into \eqref{eq:27}, and
setting $\gamma$ to $h^\#$, we obtain
\begin{equation}
  \label{eq:27d}
  \int_\M |\nabla h^\#(x,0)|^2 dx + \int_\M |h^\#(x,0)|^2 dx =
  \dfrac{1}{|\M|} \left( \int_\M h^\#(x,0)dx\right)^2.
\end{equation}
It is an easy exercise to show that
\begin{equation}
  \label{eq:27e}
\dfrac{1}{|\M|} \left( \int_\M h^\#(x,\tau)dx\right)^2 \leq \int_\M
|h^\#(x,\tau)|^2dx, \qquad \forall\,\tau.
\end{equation}
In view of \eqref{eq:27e}, the relation \eqref{eq:27d} is only
possible if 
\begin{equation*}
  |\nabla h^\#(\cdot,0)|_{L^2(\M)} = |h^\#(\cdot,0)|_{L^2(\M)} = 0.
\end{equation*}
Hence,
\begin{equation*}
  \psi |_{t=0} = \psi_0 \qquad\textrm{ in } H^1(\M).
\end{equation*}

We formally summarize these results in the following lemma.
\begin{lemma}\label{lem:weak-cont}
  The solution $\psi$ to the weak formulation
  \eqref{eq:17}, if it exists, is weakly continuous in the following sense,
  \begin{subequations}
  \begin{align}
  &\int_\M (\Delta\psi -\psi)\gamma dx \textrm{ is continuous in time
    for every } \gamma\in H^1_0(\M),\label{wc-1}\\
  &\int_M(\nabla\psi\cdot\nabla\gamma + \psi\gamma)dx \textrm{ is
    continuous in time for every } \gamma\in H^1_0(\M).\label{wc-2}
  \end{align}
  \end{subequations}
The initial condition is satisfied in the sense that 
\begin{subequations}
\begin{align}
  &\Delta\psi -\psi |_{t=0} = \Delta\psi_0 - \psi_0
  & &\textrm{in } L^2(\M)\label{wc-3}\\
  &\psi |_{t=0} = \psi_0 & &\textrm{ in } H^1(\M).\label{wc-4}
\end{align}
\end{subequations}
\end{lemma}

By virtue of Lemma \ref{lem:elliptic-reg}, any weak solutions of
\eqref{eq:17} automatically have second weak derivatives in space. In
fact, it also has second temporal-spatial cross derivatives, according
to the following lemma.
\begin{lemma}\label{lem-t-deriv}
Let $\psi(x,t)$ be a generalized solution of \eqref{eq:4},
\eqref{eq:6}--\eqref{eq:8} in the sense of \eqref{eq:17}. Then there
exists generalized derivatives $\p^2\psi/\p x\p t$ and, for any $p>1$, 
\begin{equation}
  \label{eq:28}
  \sup_{0<t<T} \|\dfrac{\p^2\psi}{\p x\p t}\|_{L^p(\M)} \leq
    Cp\sup_{0<t<T} \left(\|F\|_\lpm +
      \|\psi\|_{L^\infty(0,T;V)}\cdot\|\nabla\psi\|_\lpm \right).
\end{equation}
\end{lemma}
\begin{proof}
  We can rewrite equation \eqref{eq:4} as an elliptic equation,
  \begin{equation}
    \label{eq:29}
    (\Delta - I)\dfrac{\p}{\p t}\psi = \nabla\times F -
    \nabla\cdot\left(\nabla^\perp \psi(\Delta\psi + y -\psi)\right). 
  \end{equation}
Then, by Lemma \ref{lem:elliptic-reg1}, 
\begin{equation*}
  \left\|\dfrac{\p\psi}{\p t}\right\|_{\w1pm} \leq Cp\left(\|F\|_\lpm +
    \|\psi\|_{L^\infty(0,T: V)} \|\nabla\psi\|_\lpm\right).
\end{equation*}
Taking the supreme norm in time $t$ on the right-hand side, and then on
the left-hand side, we obtain 
\begin{equation*}
  \sup_{0<t<T} \left\|\dfrac{\p\psi}{\p t}\right\|_{\w1pm} \leq
  Cp\sup_{0<t<T}\left(\|F\|_\lpm + 
    \|\psi\|_{L^\infty(0,T: V)} \|\nabla\psi\|_\lpm\right).
\end{equation*}
\end{proof}

Finally, we are in a position to address the uniqueness of the
generalized solution of \eqref{eq:17}. 
\begin{theorem}\label{uniqueness}
  The generalized solution to the barotropic QG equation \eqref{eq:4},
  \eqref{eq:6}--\eqref{eq:8} in the sense of \eqref{eq:17}, if exists,
  must be unique.
\end{theorem}
\begin{proof}
  We let $\psi^1$ and $\psi^2$ be two solutions to the weak problem
  for the same initial data $\psi_0$. Then, for any $t\in[0,\,T]$ and
  an 
  arbitrary $\varphi\in C^\infty(Q_t)$ with $\varphi|_{\p\M} =
  \varphi(\cdot,t) = 0$, $\psi^1$ and $\psi^2$ satisfy the following
  equations, respectively,
\begin{multline}
  \label{eq:30}
  \int_\M (\Delta\psi_0 - \psi_0)\varphi(x,0) dx - \int_0^t\int_\M
  (\Delta\psi^1 - \psi^1)\dfrac{\p\varphi}{\p t}dx dt \\
   - \int_0^T\int_\M
  (\Delta\psi^1 + y - \psi^1)\nabla^\perp\psi^1\cdot\nabla\varphi dxdt =
  \int_0^T\int_\M f\varphi dxdt,
\end{multline}
\begin{multline}
  \label{eq:31}
  \int_\M (\Delta\psi_0 - \psi_0)\varphi(x,0) dx - \int_0^t\int_\M
  (\Delta\psi^2 - \psi^2)\dfrac{\p\varphi}{\p t}dx dt \\
   - \int_0^T\int_\M
  (\Delta\psi^2 + y - \psi^2)\nabla^\perp\psi^2\cdot\nabla\varphi dxdt =
  \int_0^T\int_\M f\varphi dxdt.
\end{multline}
Subtracting these two equations, and denoting $h = \psi^1 - \psi^2$,
we obtain
\begin{multline}
  \label{eq:32}
- \int_0^t\int_\M
  (\Delta h - h)\dfrac{\p\varphi}{\p t}dx dt 
   - \int_0^T\int_\M
  (\Delta h  - h)\nabla^\perp\psi^1\cdot\nabla\varphi dxdt\\
  + \int_0^T\int_\M
  (\Delta \psi^2  +y - \psi^2)\nabla^\perp h \cdot\nabla\varphi dxdt = 0.
\end{multline}
An integration by parts in space in the first term leads to
\begin{multline}
  \label{eq:32a}
\int_0^t\int_\M
  (\nabla h\cdot\nabla\p_t\varphi + h\p_t\varphi)dx dt 
   - \int_0^T\int_\M
  (\Delta h  - h)\nabla^\perp\psi^1\cdot\nabla\varphi dxdt\\
  + \int_0^T\int_\M
  (\Delta \psi^2  +y - \psi^2)\nabla^\perp h \cdot\nabla\varphi dxdt = 0.
\end{multline}
Both $\varphi^1$ and $\varphi^2$ assume space-independent values on
the boundary $\p\M$, and so does the difference $h$ between
them. Thus, after a shifting in the vertical direction, $h$ will
vanish on the boundary. We denote this shifted function by $h^\#\in
L^\infty(0,t; H^1_0(\M))$. $h$ and $h^\#$ are related via
\begin{equation}
  \label{eq:33}
  h(x,\tau) = h^\#(x,\tau) + l(\tau),\qquad 0\leq \tau\leq t
\end{equation}
for some function $l(\tau)$. Both $\psi^1$ and $\psi^2$ have a zero
average over $\M$, and so does their difference $h$. Integrating
\eqref{eq:33} over $\M$ we establish a simple relation between $l$
and $h^\#$,
\begin{equation}
  \label{eq:34}
  l(\tau) = -\dfrac{1}{|\M|}\int_\M h^\#(x,\tau)dx.
\end{equation}
Since both $\psi^1$ and $\psi^2$ satisfy the same initial condition
\eqref{eq:8} in the sense of \eqref{wc-4}, it is easy to see that
\begin{equation}
  \label{eq:34c}
  h(\cdot,0) = h^\#(\cdot, 0) = 0.
\end{equation}

Replacing $h$ by $h^\#+l$ in the first and third integrals of
\eqref{eq:32a} yields
 \begin{multline}
  \label{eq:35}
\int_0^t\int_\M
  (\nabla h^\#\cdot\nabla\p_t\varphi + h^\#\p_t\varphi)dx d\tau + \\
  \int_0^t\int_M l\p_t\varphi dxd\tau 
   - \int_0^T\int_\M
  (\Delta h  - h)\nabla^\perp\psi^1\cdot\nabla\varphi dxd\tau \\
  + \int_0^T\int_\M
  (\Delta \psi^2  +y - \psi^2)\nabla^\perp h^\# \cdot\nabla\varphi dxd\tau = 0.
\end{multline}
In view of the regularity results in the previous lemma, and the facts
that $h^\#(\cdot,0) = h(\cdot,0) = 0$ and $\varphi(\cdot,t) = 0$, we
integrate by parts in time in \eqref{eq:35} and arrive at
 \begin{multline}
  \label{eq:36}
-\int_0^t\int_\M
  (\p_t\nabla h^\#\cdot\nabla\varphi + \p_t h^\#\varphi)dx d\tau -\\
  \int_0^t\int_M \p_t l\varphi dxd\tau 
   - \int_0^T\int_\M
  (\Delta h  - h)\nabla^\perp\psi^1\cdot\nabla\varphi dxd\tau \\
  + \int_0^T\int_\M
  (\Delta \psi^2  +y - \psi^2)\nabla^\perp h^\# \cdot\nabla\varphi dxd\tau = 0.
\end{multline}
We note that each of the integrals is linear and continuous with
respect to $\varphi$ in the norm of $L^2(0,T: H^1_0(\M))$. Thus, we
can let $\varphi$ tend to $h^\#$ in $L^2(0,T; H^1_0(\M))$, pass to
the limit in \eqref{eq:36}, and notice the fact that $\nabla h^\#\cdot
\nabla^\perp h^\# = 0$, we obtain
 \begin{multline*}
-\int_0^t\int_\M
  (\p_t\nabla h^\#\cdot\nabla h^\# + \p_t h^\# h^\#)dx d\tau -
  \int_0^t\int_M \p_t l h^\# dxd\tau  \\
   - \int_0^T\int_\M
  (\Delta h  - h)\nabla^\perp\psi^1\cdot\nabla h^\# dxd\tau = 0.
\end{multline*}
\begin{multline}
\label{eq:37}
-\int_0^t\dfrac{1}{2}\dfrac{d}{dt}\|h^\#\|^2_{H^1_0(\M)}dt d\tau -
  \int_0^t\dfrac{1}{2}\dfrac{d}{dt}\|h^\#\|^2_{L^2(\M)} d\tau \\
   - \int_0^t\p_t l\int_\M
  h^\# dxd\tau    - \int_0^t\int_\M
  (\Delta h  - h)\nabla^\perp\psi^1\cdot\nabla h^\# dxd\tau = 0.
\end{multline}
In the above, $\|\cdot\|_{H^1_0(\M)} \equiv \|\nabla(\cdot)\|_{L^2(\M)}$. 
Using the expression in \eqref{eq:34} for $l(\tau)$, we derive that
\begin{multline}
\label{eq:38}
-\dfrac{1}{2}\|h^\#(\cdot,t)\|^2_{H^1_0(\M)} d\tau -
  \dfrac{1}{2}\|h^\#(\cdot,t)\|^2_{L^2(\M)} d\tau 
   + \dfrac{1}{|\M|} \int_0^t \dfrac{1}{2}\p_t \left(\int_\M
  h^\# dx\right)^2 d\tau    \\
  - \int_0^t\int_\M
  (\Delta h  - h)\nabla^\perp\psi^1\cdot\nabla h^\# dxd\tau = 0.
\end{multline}
\begin{multline}
\label{eq:39}
\|h^\#(\cdot,t)\|^2_{H^1_0(\M)} + \|h^\#(\cdot,t)\|^2_{L^2(\M)}  
   - \dfrac{1}{|\M|} \left(\int_\M
  h^\# dx\right)^2   =  \\
  - 2\int_0^t\int_\M
  (\Delta h  - h)\nabla^\perp\psi^1\cdot\nabla h^\# dxd\tau.
\end{multline}
Using the estimate \eqref{eq:27e} for $\tau=t$ in the above, 
one derives that
\begin{equation}
  \label{eq:40}
  \|h^\#(\cdot,t)\|_{H^1_0(\M)}^2 \leq -2 \int_0^t\int_\M (\Delta h -h
  )\nabla^\perp \psi^1\cdot\nabla h^\# dx d\tau.
\end{equation}
First, we note that, with the change of variable \eqref{eq:33}, 
\begin{equation*}
  \Delta h - h = \Delta h^\# - h^\# - l(\tau).
\end{equation*}
Thanks to the non-penetration boundary conditions on $\p\M$, we have
\begin{align*}
  &\int_\M h^\#\nabla^\perp \psi^1 \cdot \nabla h^\# dx = 0,\\
  &\int_\M l(\tau) \nabla^\perp \psi^1 \cdot \nabla h^\# dx = 0.
\end{align*}
Therefore, 
\begin{equation*}
  \int_\M (\Delta h -h
  )\nabla^\perp \psi^1\cdot\nabla h^\# dx = \int_\M \Delta
  h^\#\nabla^\perp\psi^1 \cdot\nabla h^\# dx,
\end{equation*}
and \eqref{eq:40} becomes
\begin{equation}
  \label{eq:41}
  \|h^\#(\cdot,t)\|_{H^1_0(\M)}^2 \leq -2 \int_0^t\int_\M \Delta h^\#
  \nabla^\perp \psi^1 \cdot\nabla h^\# dx d\tau.
\end{equation}
To further investigate the integral on the right-hand side, we
adopt the notations
\begin{equation*}
  (u_1,\,u_2) = \nabla^\perp \psi^1 = (-\p_2\psi^1,\,\p_1\psi^1).
\end{equation*}
Then, using the Einstein convention of repeated indices for $1\le
i,\,j\le 2$ and the fact that 
$\nabla^\perp\psi^1\cdot\nabla h^\#$ vanishes on the boundary $\p\M$,
we proceed by integration by parts and obtain
\begin{align*}
  \int_\M \Delta h^\# \nabla^\perp\psi^1\cdot\nabla h^\# dx 
  &= \int_\M \p_i^2 h^\# u_j\p_j h^\# dx \\
  &= -\int_\M \left(\p_i h^\# \p_i u_j \p_j h^\# +  
     \p_i h^\# u_j \p_j\p_i h^\#\right) dx\\
  &= - \int_\M \p_i h^\#\p_i u_j \p_j h^\# dx.
\end{align*}
Reverting back to the standard index conventions and after rearrangements,
we find that
\begin{multline}
  \label{eq:42}
  \int_\M \Delta h^\# \nabla^\perp\psi^1\cdot\nabla h^\# dx = \\-\int_\M
  \left\{ \left[-(\p_1 h^\#)^2 + (\p_2 h^\#)^2\right]\p_1\p_2\psi^1 +
    \p_1 h^\#\p_2 h^\#(\p_1^2\psi^1 - \p_2^2\psi^1)\right\} dx.
\end{multline}
Hence, we have
\begin{align*}
  &\| h^\#(\cdot,t)\|^2_{H^1_0(\M)} \\
\leq &2\int_0^t \int_\M \left\{ \left[-(\p_1 h^\#)^2 + (\p_2
    h^\#)^2\right]\p_1\p_2\psi^1 + \p_1 h^\#\p_2 h^\#(\p_1^2\psi^1 -
    \p_2^2\psi^1)\right\} dx d\tau\\ 
\leq &\int_0^t \int_\M \left\{ 2\left(|\p_1 h^\#|^2 + |\p_2
    h^\#|^2\right)\cdot|\p_1\p_2\psi^1| + (|\p_1 h^\#|^2 + |\p_2
       h^\#|^2)(|\p_1^2\psi^1| + |\p_2^2\psi^1|)\right\} dx d\tau\\
  \leq &\int_0^t \int_\M |\nabla h^\#|^2 \left( |\p_1^2\psi^1| +
        |\p_2^2\psi^1|+2|\p_1\p_2\psi^1| \right) dx d\tau.
\end{align*}
We note that, according to Lemma \ref{lem:elliptic-reg}, both
$\nabla\psi^1$ and $\nabla\psi^2$ are H\"older continuous in $\M$ and
essentially bounded in $(0,T)$, and so is $\nabla h^\#$. We let
$0<\epsilon < 1$ be an arbitrary parameter. Then, we have
\begin{multline*}
\| h^\#(\cdot,t)\|^2_{H^1_0(\M)} 
    \leq 
   |\nabla h^\#|^\epsilon_{L^\infty(Q_T)} \int_0^t \int_\M |\nabla
  h^\#|^{2-\epsilon} \left( |\p_1^2\psi^1| + 
        |\p_2^2\psi^1|+2|\p_1\p_2\psi^1| \right) dx d\tau.
\end{multline*}
Applying H\"older's inequality to the spatial integral on the
right-hand side, we obtain
\begin{equation}
\label{eq:43}
\| h^\#(\cdot,t)\|^2_{H^1_0(\M)} 
    \leq 
   C|\nabla h^\#|^\epsilon_{L^\infty(Q_T)} \int_0^t
   \|h^\#(\cdot,\tau)\|^{2-\epsilon}_{\h10m} \cdot
   \|\psi^1(\cdot,\tau)\|_{W^{2,\frac{2}{\epsilon}}(\M)} d\tau.
\end{equation}
We note that both $\psi^1$ and $\psi^2$ belongs to $L^\infty(0,T; V)$,
and so does $h$. By Lemma \ref{lem:elliptic-reg}, 
$\psi^1(\cdot,\tau)\in W^{2,p}(\M)$ with $p = 2/\epsilon$, and $\nabla
h^\#$, which equals $\nabla h$, is H\"older continuous, for
a.e.~$\tau$. We set
\begin{align*}
  M_1 =& \sup_{0<t<T} \|\psi^1(\cdot,t)\|_V,\\
  M_2 =& \sup_{0<t<T} \|h^1(\cdot,t)\|_V.
\end{align*}
Then, it is inferred from \eqref{eq:14} and \eqref{eq:15} that 
\begin{align*}
  &\sup_{0<t<T} \| \psi^1(\cdot,t)\|_{W^{2,\frac{2}{\epsilon}}(\M)}
  \leq C\dfrac{2}{\epsilon} M_1,\\
  &\sup_{0<t<T} \| \nabla h^\# (\cdot,t)\|_{\linfm}
  \leq C M_2,
\end{align*}
where $C$ designates generic constants that are independent of
$\psi^1$, $\psi^2$ and $\epsilon$. Using these estimates in
\eqref{eq:43} leads to 
\begin{equation}
\label{eq:44}
\| h^\#(\cdot,t)\|^2_{H^1_0(\M)} 
    \leq 
   \dfrac{C}{\epsilon} M_1 M_2^\epsilon \int_0^t
   \|h^\#(\cdot,\tau)\|^{2-\epsilon}_{\h10m} d\tau.
\end{equation}
We denote 
\begin{equation*}
  \sigma(\cdot,t) \equiv \| h^\#(\cdot, t)\|_{\h10m}.
\end{equation*}
Then \eqref{eq:44} can be written as 
\begin{equation}
  \label{eq:45}
  \sigma^2(t) \leq \dfrac{C}{\epsilon} M_1 M^\epsilon_2\int_0^t
  \sigma^{2-\epsilon}(\tau) d\tau.
\end{equation}
An estimate on $\sigma$ can be obtained by  the Gronwall
inequality. Indeed, we let 
\begin{equation*}
  F(t) = \dfrac{C}{\epsilon} M_1 M^\epsilon_2\int_0^t
  \sigma^{2-\epsilon}(\tau) d\tau.
\end{equation*}
Taking derivative of this function and using \eqref{eq:45}, we find
\begin{equation*}
  \dfrac{d}{dt} F(t) \leq \dfrac{C}{\epsilon} M_1 M^\epsilon_2
  F^{1-\frac{\epsilon}{2}}(t). 
\end{equation*}
Integration of this inequality yields
\begin{equation*}
  F(t) \leq (CM_1 t)^{\frac{2}{\epsilon}} M_2^2. 
\end{equation*}
Thus,
\begin{equation}
  \label{eq:46}
  \| h^\#(\cdot,t)\|^2_{\h10m} \leq F(t) \leq (C M_1
  t)^{\frac{2}{\epsilon}} M^2_2. 
\end{equation}
We take 
\begin{equation*}
  t^\ast = \dfrac{1}{2CM_1}.
\end{equation*}
Then, for $0\leq t\leq t^\ast$, 
\begin{equation}
  \label{eq:47}
  \| h^\#(\cdot,t)\|^2_{\h10m} \leq
  \left(\dfrac{1}{2}\right)^{\frac{2}{\epsilon}} M^2_2.  
\end{equation}
This estimate holds for arbitrary $\epsilon > 0$. Thus, $  \|
h^\#(\cdot,t)\|^2_{\h10m} $ must vanish for $0\leq t\leq t^\ast$. The
process can be repeated over subsequent time intervals of length
$t^\ast$, and thus $  \| h^\#(\cdot,t)\|^2_{\h10m} = 0$ for the whole
time interval $[0,\,T]$. Combined with the relations \eqref{eq:33} and
\eqref{eq:34}, it implies that 
\begin{equation*}
  h(\cdot,t) = 0\qquad \textrm{for a.e.} 0\leq t\leq T.
\end{equation*}
The generalized solution to the barotropic QG equation \eqref{eq:4},
\eqref{eq:6}--\eqref{eq:8} must be unique.
\end{proof}

\section{Existence of a solution to the weak problem}
We establish the existence of a solution to the weak problem
\eqref{eq:17} through an iterative scheme. 
To get started, we set
\begin{equation}
  \label{eq:48}
 q^0(x,t) = q_0(x),
\end{equation}
where $q_0$ is the initial QGPV computed from the initial
streamfunction $\psi_0$. 
We assume that $q^n$, for a $n\ge 0$, is known, we compute $q^{n+1}$
as follows. First, the streamfunction $\psi^n$ corresponding to $q^n$
is obtained from the non-standard elliptic BVP,
\begin{subequations}
  \label{eq:49}
  \begin{align}
    \Delta\psi^n - \psi^n + \beta y &= q^n, & &\M,\label{eq:49a}\\
    \psi^n &= l^n, & &\p\M,\label{eq:49b}\\
    \int_\M \psi^n dx &= 0.\label{eq:49c} & &
  \end{align}
\end{subequations}
The corresponding velocity field $\ub^n$ is obtained through
\begin{equation}
  \label{eq:50}
  \ub^n = \nabla^\perp \psi^n. 
\end{equation}
From this velocity field, a flow mapping (or the trajectory path) $\Phi^n_t(a)$ for each
$a\in\M$ is constructed,
\begin{subequations}
  \label{eq:51}
  \begin{align}
    \dfrac{d}{dt}\Phi^n_t(a) &= \ub^n(\Phi^n_t(a), t),\label{eq:51a}\\
    \Phi^n_0(a) &= a.\label{eq:51b}
  \end{align}
\end{subequations}
Finally, the QGPV field is updated using the flow mapping,
\begin{equation}
  \label{eq:52a}
  q^{n+1}(\Phi^n_t(a), t) = q_0(a) + \int_0^t
  f(\Phi^n_s(a),s)ds, \qquad \forall a\in\M,
\end{equation}
or, equivalently,
\begin{equation}
  \label{eq:52b}
  q^{n+1}(x, t) = q_0(\Phi^n_{-t}(x)) + \int_0^t
  f(\Phi^n_{s-t}(x),s)ds,\qquad \forall x\in\M.
\end{equation}

The iterative scheme \eqref{eq:48}--\eqref{eq:52b} is straightforward
except for the solution of the initial value problem
\eqref{eq:51}. The elegant argument presented in \cite{Kato1967-cj}
requires $\ub$ to be $C^1$, which is not the case here. Given that
$q^n(\cdot, t)\in L^\infty(\M)$, $u^n$ is quasi-Lipschitz according to
Lemma \ref{lem:elliptic-reg}. It is well known that (see \cite{Marchioro1994-yt}) the
quasi-Lipschitz continuity is also {\itshape sufficient} to ensure the
global existence and uniqueness of a solution to the
IVP. Specifically, the following result is available.

\begin{lemma}\label{lem:flow}
  Assume the velocity field $\ub^n$ is uniformly bounded in $Q_T$ and
  quasi-Lipschitz continuous in $\M$, i.e.~$\ub^n(x,t) \le C$ for
  $\forall\, (x,t)\in Q_T$, $|\ub^n(x_1,t) - \ub^n(x_2,t)|\leq
  C\chi(|x_1 - x_2|)$, with $C$ independent of $x$ or $t$. Then
  the initial value problem \eqref{eq:51} has a global unique
  solution. 
\end{lemma}

The proof is recalled in Appendix \ref{ap-1}. We note that, since
$\ub^n$ is divergence free, the corresponding flow mapping preserves
area, by virtue of the Liouville Theorem (see
e.g.~\cite{Hartman2002-pg}). 

We now show that the sequence $(q^n, \,\psi^n,\, \ub^n)$ generated by
the iterative scheme converges, and the limit solves the weak problem
\eqref{eq:17}. We follow the procedure laid out in
\cite{Marchioro1994-yt}, and start with the convergence of the flow
mapping $\Phi^n_t(\cdot)$. 
\begin{lemma}
  As $n\longrightarrow \infty$,
  \begin{equation*}
    \Phi^n_t(a) \longrightarrow \Phi_t(a) \qquad \textrm{strongly in }
    L^\infty(0,T; L^1(\M)). 
  \end{equation*}
\end{lemma}
\begin{proof}
We first write the IVP \eqref{eq:51} in the integral form,
From \eqref{eq:54}, 
  \begin{equation}
    \label{eq:70}
  \Phi^n_t(a) = a + \int_0^t \ub^n(\Phi^n_s(a),s)ds.
  \end{equation}
We subtract $\Phi^n_t(a)$ and $\Phi^{n-1}_t(a)$,
  \begin{equation}
    \label{eq:71}
  \Phi^n_t(a) - \Phi^{n-1}_t(a) = \int_0^t \left(
    \ub^n(\Phi^n_s(a),s) - \ub^n(\Phi^n_s(a),s) \right)ds.
  \end{equation}
We define 
\begin{equation}
  \label{eq:72}
  K(x,y) = \nabla_x^\perp G(x,y),
\end{equation}
where $G(x,y)$ is the Green's function for $\Delta - I$ with Dirichlet
boundary conditions on $\M$. 
Then, using the relation \eqref{eq:13e},
\begin{multline}
\label{eq:u-n}
  \ub^n(\Phi^n_s(a),s) = \nabla^\perp \psi_n(\Phi^n_s(a),s)  =\\
  \int_\M K(\Phi^n_s(a),y)(q^n(y,s) -y_2)dy + l(q^n)\int_\M
  K(\Phi^n_s(a),y)dy. 
\end{multline}
Thus, 
\begin{align*}
  &\ub^n(\Phi^n_s(a),s) - \ub^{n-1}(\Phi^{n-1}_s(a),s) \\
  =& \int_\M \left( K(\Phi^n_s(a),y)q^n(y,s) -
     K(\Phi^{n-1}_s(a),y)q^{n-1}(y,s)\right) dy - \\
    &\int_\M \left(
     K(\Phi^n_s(a),y) - 
     K(\Phi^{n-1}_s(a),y)\right)y_2 dy + \\
    &\int_\M \left(
     l(q^n)K(\Phi^n_s(a),y) - 
     l(q^{n-1})K(\Phi^{n-1}_s(a),y)\right)y_2 dy 
\end{align*}
Using \eqref{eq:52a}, the difference between
$K(\Phi^n_s(a),y)q^n(y,s)$ and 
$K(\Phi^{n-1}_s(a),y)q^{n-1}(y,s)$  can be split into three parts, 
\begin{multline*}
\int_\M\left(  K(\Phi^n_s(a),y)q^n(y,s) -
  K(\Phi^{n-1}_s(a),y)q^{n-1}(y,s)\right)dy = \\
\int_\M\left(  K(\Phi^n_s(a),y)-
  K(\Phi^{n-1}_s(a),y)\right) q^{n}(y,s) dy + \\
\int_\M  K(\Phi^{n-1}_s(a),y)\left( q_0(\Phi^{n-1}_{-s}(y)) -
  q_0(\Phi^{n-2}_{-s}(y))\right) dy + \\ 
\int_\M  K(\Phi^{n-1}_s(a),y) \int_0^s \left(f(\Phi^{n-1}_{\tau-s}(y),\tau) -
  f(\Phi^{n-2}_{\tau-s}(y), \tau)\right) d\tau dy.
\end{multline*}
Thus, the difference in the velocity fields $u^n(\Phi^n_s(a),s)$ and
$\ub^{n-1}(\Phi^{n-1}_s(a),s)$ is split into four parts, labeled as
$I-IV$ below,
\begin{multline}
\label{eq:72aaa}
    \ub^n(\Phi^n_s(a),s) - \ub^{n-1}(\Phi^{n-1}_s(a),s) = \\
\int_\M\left(  K(\Phi^n_s(a),y)-
  K(\Phi^{n-1}_s(a),y)\right) \left(q^{n}(y,s) - y_2 + l(q^n)\right)
dy + \\
\int_\M  K(\Phi^{n-1}_s(a),y)\left( q_0(\Phi^{n-1}_{-s}(y)) -
  q_0(\Phi^{n-2}_{-s}(y))\right) dy + \\ 
\int_\M  K(\Phi^{n-1}_s(a),y) \int_0^s \left(f(\Phi^{n-1}_{\tau-s}(y),\tau) -
  f(\Phi^{n-2}_{\tau-s}(y), \tau)\right) d\tau dy + \\
\int_\M  K(\Phi^{n-1}_s(a),y)\left( l(q^n) - l(q^{n-1})\right)dy \\
:= \textrm{I} + \textrm{II} + \textrm{III} + \textrm{IV}.
\end{multline}
We now estimate each of these four parts. For the first term $I$, it
is clear from the proof of Lemma \eqref{lem:elliptic-reg} that
$\int_\M K(x,y)dy$ is quasi-Lipschitz continuous. Also using the fact
that $q^n$ is essentially bounded, we proceed, 
\begin{align*}
  |I| &\leq |q^n-y_2 + l|_\infty \int_\M \left| K(\Phi^n_s(a), y) -
        K(\Phi^{n-1}_s(a),y)\right| dy \\
  &\leq C|q^n-y_2 + l|_\infty \left( 1+\left|\ln |\Phi^n_s(a) - \Phi^{n-1}_s(a)|\right|\right)
    \cdot |\Phi^n_s(a) - \Phi^{n-1}_s(a)|.
  \end{align*}
We define
\begin{equation}
\label{eq:72a}
  \delta^n_s(a) = \left|\Phi^n_s(a) - \Phi^{n-1}_s(a)\right|.
\end{equation}
Then, for the term $I$ we have
\begin{equation}
  \label{eq:72aa}
 |I| \leq  C(\M, |q_0|_\infty, |f|_\infty) \chi\left( \delta^n_s(a)\right),
\end{equation}
We apply change of variables in II, utilizing that fact that the
mapping $y = \Phi^n_s(b)$ is area preserving,
\begin{align*}
  II &= \int_\M K\left(\Phi^{n-1}_s(a),y\right)
       q_0\left(\Phi^{n-1}_{-s}(y)\right) dy - \int_\M
       K\left(\Phi^{n-1}_s(a),y\right) 
       q_0\left(\Phi^{n-2}_{-s}(y)\right) db\\
&= \int_\M K\left(\Phi^{n-1}_s(a),\Phi^{n-1}_s(b)\right))
       q_0\left(b \right) db - \int_\M
       K\left(\Phi^{n-1}_s(a),\Phi^{n-2}_s(b)\right) 
       q_0\left( b \right) db\\
&= \int_\M \left(K\left(\Phi^{n-1}_s(a),\Phi^{n-1}_s(b)\right))-
              K\left(\Phi^{n-1}_s(a),\Phi^{n-2}_s(b)\right) \right)
       q_0\left( b \right) db.
\end{align*}
Integrating $II$ in $a$ over $\M$, we find that
\begin{align*}
  \int_\M |II|da \leq |q_0|_\infty \int_\M \int_\M \left|K\left(x,\Phi^{n-1}_s(b)\right))-
              K\left(x,\Phi^{n-2}_s(b)\right) \right| dx db
\end{align*}
Again, with the quasi-Lipschitz continuity of $\int_\M K(x,y)dy$, we
derive that
\begin{equation}
  \label{eq:73}
  \int_\M |II|da \leq |q_0|_\infty \int_\M \chi(\delta^{n-1}_s(b) )db. 
\end{equation}
Also applying the change of variable and integration to $III$, 
\begin{align*}
  III &= \int_0^s\int_\M K\left(\Phi^{n-1}_s(a),y\right)
       f\left(\Phi^{n-1}_{\tau-s}(y),\tau\right) dyd\tau - \int_0^s \int_\M
       K\left(\Phi^{n-1}_s(a),y\right) 
       f\left(\Phi^{n-2}_{\tau-s}(y),\tau\right) dyd\tau\\
&= \int_0^s \int_\M K\left(\Phi^{n-1}_s(a),\Phi^{n-1}_{s-\tau}(a)\right)
       f\left(b, \tau \right) dbd\tau - \int_0^s \int_\M
       K\left(\Phi^{n-1}_s(a),\Phi^{n-2}_{s-\tau}(b)\right) 
       f\left( b, \tau \right) db d\tau\\
&= \int_0^s \int_\M \left(K\left(\Phi^{n-1}_s(a),\Phi^{n-1}_{s-\tau}(b)\right))-
              K\left(\Phi^{n-1}_s(a),\Phi^{n-2}_{s-\tau}(b)\right) \right)
       f\left( b, \tau \right) dbd \tau.
\end{align*}
By integrating $|III|$ over $\M$, and, again using the quasi-Lipschitz
continuity of $\int_\M K(x,y)dy$, we reach
\begin{equation}
  \label{eq:74}
  \int_\M |III|da \leq |f|_\infty \int_\M \int_0^s \chi(\delta^{n-1}_{s-\tau}(b))d\tau db. 
\end{equation}

Using the formula \eqref{eq:13d} for $l(q)$, we find that
\begin{equation}
  \label{eq:75}
     l(q^n) - l(q^{n-1}) = -\dfrac{\int_\M \int_\M G(x,y)(q^n(y,s)-q^{n-1}(y,s)) dydx}{|\M| + \int_\M
     \int_\M G(x,y)dydx}.
\end{equation}
Using the formula \eqref{eq:52b} for $q^n$ and $q^{n-1}$ in the above, and
the mapping provided by $\Phi^n_t$ and $\Phi^{n-1}_t$, we obtain
that, after some changes of variables,
\begin{multline*}
     l(q^n) - l(q^{n-1}) = -\dfrac{1}{|\M| + \int_\M\int_\M G(x,y)dydx}\cdot\\
     \left(\int_\M \int_\M q_0(a) \left(G(x,\Phi^{n-1}_s(a)) -
         G(x,\Phi^{n-2}_s(a))\right) dx\,da +\right. \\
     \left.\int_\M \int_\M\int_0^s  \left(G(x,\Phi^{n-1}_{s-\tau} (a)) -
         G(x,\Phi^{n-2}_{s-\tau}(a))\right)f(a,\tau)d\tau dxda\right).
\end{multline*}
Using the assumptions that $q_0$ and $f$ are essentially bounded, in
$\M$ and $Q_T$, respectively, we obtain
\begin{multline*}
    \left| l(q^n) - l(q^{n-1})\right| = \dfrac{|q_0|_\infty + |f|_\infty}{|\M| + \int_\M\int_\M G(x,y)dydx}\cdot\\
     \left(\int_\M \int_\M \left|G(x,\Phi^{n-1}_s(a)) -
         G(x,\Phi^{n-2}_s(a))\right| dx\,da +\right. \\
     \left.\int_\M \int_\M\int_0^s  \left|G(x,\Phi^{n-1}_{s-\tau}(a)) -
         G(x,\Phi^{n-2}_{s-\tau}(a))\right| d\tau dxda\right).
\end{multline*}
It is also clear from the proof of Lemma \ref{lem:elliptic-reg} that
$\int_\M G(x,y)dx$ is $C^1$. Therefore, we have that
\begin{multline*}
    \left| l(q^n) - l(q^{n-1})\right| \leq C(\M, q_0, f)
     \left(\int_\M \left|\Phi^{n-1}_s(a) -
         \Phi^{n-2}_s(a)\right| da +\right. \\
     \left.\int_\M\int_0^s  \left|\Phi^{n-1}_{s-\tau}(a) -
         \Phi^{n-2}_{s-\tau}(a)\right| d\tau da\right).
\end{multline*}
Using the notation introduced in \eqref{eq:72a}, one can write the
above estimate as
\begin{equation}
  \label{eq:76}
    \left| l(q^n) - l(q^{n-1})\right| \leq C(\M, q_0, f) \left(
      \int_\M \delta^{n-1}_s(a)da + \int_0^s\int_\M
      \delta^{n-1}_{s-\tau}(a) da d\tau\right).
\end{equation}
Thus,
\begin{equation*}
  |IV| \leq C(\M, |q_0|_\infty, |f|_\infty)\left(
      \int_\M \delta^{n-1}_s(a)da + \int_0^s\int_\M
      \delta^{n-1}_{s-\tau} da d\tau\right)\cdot \int_\M
    K(\Phi^{n-1}_s(a), y)dy
\end{equation*}
In the analysis leading to the estimate \eqref{eq:14b}, it is clear
that the last integral on the right-hand side is uniformly
bounded, and the bound depends on the domain $\M$ only. Therefore, we
can simply write the above relation as
\begin{equation}
  \label{eq:77}
  |IV| \leq C(\M, |q_0|_\infty, |f|_\infty)\left(
      \int_\M \delta^{n-1}_s(a)da + \int_0^s\int_\M
      \delta^{n-1}_{s-\tau}(a) da\, d\tau\right). 
\end{equation}

Using the relation \eqref{eq:72aaa} and the estimates \eqref{eq:72aa},
\eqref{eq:73}, \eqref{eq:74}, and \eqref{eq:77} in \eqref{eq:71}, we obtain
\begin{multline}
  \label{eq:78}
 \int_\M \left|\Phi^n_t(a) - \Phi^{n-1}_t(a)\right|da \leq 
C(\M, |q_0|_\infty, |f|_\infty) \int_0^t \left( \int_\M \chi(\delta^n_s(a)) da\right. + \\
  \left.\int_\M \chi(\delta^{n-1}_s(a))da + \int_0^s\int_\M
  \chi(\delta^{n-1}_{s-\tau}(a) da d\tau + \right.\\
  \left.\int_\M
  \delta^{n-1}_s(a)da + \int_0^s \int_\M
  \delta^{n-1}_{s-\tau}(a)dad\tau \right) ds.
\end{multline}
We define
\begin{equation}
  \label{eq:79}
  \rho^n_t = \dfrac{1}{|\M|}  \int_\M \left|\Phi^n_t(a) -
    \Phi^{n-1}_t(a)\right|da \equiv \dfrac{1}{|\M|}  \int_\M
  \delta^n_t(a) da. 
\end{equation}
We note that, thanks to the convexity of $\chi$, 
\begin{equation}
  \label{eq:80}
  \dfrac{1}{|\M|}  \int_\M\chi\left(\delta^n_t(a)\right)  da \leq
  \chi\left(  \dfrac{1}{|\M|} \int_\M\delta^n_t(a) da\right)  \equiv
  \chi(\rho^n_s). 
\end{equation}
Then \eqref{eq:78} can be written as 
\begin{equation}
  \label{eq:81}
  \rho^n_t \leq C(\M, |q_0|_\infty, |f|_\infty) \int_0^t \left( \chi(\rho^n_s) +
    \chi(\rho^{n-1}_s) + \rho^{n-1}_s + \int_s
    \chi(\rho^{n-1}_{s-\tau})d\tau + \int_0^s \rho^{n-1}_{s-\tau}
    d\tau\right) ds.
\end{equation}
For an arbitrary $n>0$, we define
\begin{equation}
\label{eq:81a}
  e^n(t) = \sup_{k\ge n} \rho^k_t. 
\end{equation}
Therefore, for a fixed $t$, $\rho^n_t \leq e^n(t)$, and $e^n(t)$ is
monotonically decreasing in 
$n$. We also note that, since $\chi(\cdot)$ is a monotonically
increasing function, 
\begin{equation*}
  \chi(\rho^n_s) \leq \chi(e^n(s)) \leq \chi(e^{n-1}(s)). 
\end{equation*}
Using these relations in \eqref{eq:81}, one finds that 
\begin{multline}
  \label{eq:82}
  e^n(t) \leq C(\M, |q_0|_\infty, |f|_\infty) \int_0^t \left( 2\chi(e^{n-1}(s))  +
     e^{n-1}(s)\right. \\+ \left.\int_0^s
    \chi(e^{n-1}({s-\tau}))d\tau + \int_0^s e^{n-1}(s-\tau)
    d\tau\right) ds.
\end{multline}
For an arbitrary scalar function $h(\cdot)$, we note that
\begin{equation*}
  \int_0^t \int_0^s h(s-\tau) d\tau ds = \int_0^t (t-s)h(s) ds.
\end{equation*}
Thus one can replace  the double integrals in \eqref{eq:82} with single
integrals, and after rearrangements, one obtains
\begin{equation}
  \label{eq:83}
  e^n(t) \leq C(\M, |q_0|_\infty, |f|_\infty) \int_0^t \left( (t-s+2)\chi(e^{n-1}(s))  +
     (t-s+1)e^{n-1}(s) \right) ds.
\end{equation}
Using the inequality \eqref{eq:57a} on the function $\chi$, one
obtains from \eqref{eq:83} that 
\begin{align*}
  e^n(t) &\leq C(\M, |q_0|_\infty, |f|_\infty) \int_0^t  \left((t-s+2)\left(-\ln\epsilon
      \cdot e^{n-1}(s)+\epsilon\right)    +
     \left( t-s+1 \right)e^{n-1}(s)  \right) ds\\
&\leq C(\M, |q_0|_\infty, |f|_\infty) \int_0^t  \left(3\left(-\ln\epsilon
      \cdot e^{n-1}(s)+\epsilon\right)    +
      2e^{n-1}(s)  \right) ds\\
&\leq C(\M, |q_0|_\infty, |f|_\infty) \int_0^t  \left(\left(-3\ln\epsilon + 2\right)
      \cdot e^{n-1}(s)+3\epsilon\right) ds\\
\end{align*}
Focusing on small $t$'s, say $t\in[0,\,1]$, one derive from the above
that 
\begin{equation}
  \label{eq:83a}
e^n(t) \leq C(\M, |q_0|_\infty, |f|_\infty) \left(-\ln\epsilon + 1\right)\int_0^t
  e^{n-1}(s) ds + C\epsilon t.
\end{equation}
This inequality is now very similar to the inequality \eqref{eq:58}.
When $n=1$, by the definition \eqref{eq:81a} and the relation
  \eqref{eq:71}, one derives that
  \begin{align*}
    e^1(t) = \sup_{k\ge 1} \rho^k(t) 
&= \sup_{k\ge 1}\dfrac{1}{|\M|}\int_\M \left|\Phi^k_t(a) - 
\Phi^{k-1}_t(a)\right| da\\
    &\leq \sup_{k\ge 1}\dfrac{1}{|\M|}\int_\M \int_0^t \left|\ub^k_t(\Phi^k_s(a),s) - 
\ub^{k-1}(\Phi^{k-1}_s(a),s)\right|ds da.
  \end{align*}
From above, and the uniform boundedness of $\ub^k$, one deduces that 
\begin{equation}
  \label{eq:84}
  e^1(t) \leq C(\M, |q_0|_\infty, |f|_\infty) t.
\end{equation}
We now temporarily omit the dependence of $C(\M, |q_0|_\infty, |f|_\infty)$, and
simply write $C$, for the sake of conciseness.
By induction, one can derive that
\begin{equation}
  \label{eq:85}
      e^n(t) \leq C^n
    \left(-\ln\epsilon+1\right)^{n-1} \dfrac{t^n}{n!} +
    \epsilon\sum_{k=1}^{n-1} C^k\left(-\ln\epsilon+1\right)^{k-1}
      \dfrac{t^k}{k!}. 
\end{equation}
Majorizing the last summation on the right-hand side, one obtains that 
\begin{equation}
  \label{eq:86}
      e^n(t) \leq \dfrac{C\cdot
      t\left(C(-\ln\epsilon + 1)t\right)^{n-1}}{n!} +
      C\cdot t\cdot\epsilon^{1-Ct}\cdot e^{Ct}.
\end{equation}
As $\epsilon$ can be arbitrary, we take $\epsilon = e^{-(n-1)}$, and
find that 
\begin{equation}
  \label{eq:87}
      e^n(t) \leq \dfrac{\left(C\cdot t\right)^n n^{n-1}}{n!} +
      C\cdot t\cdot e^{-n(1-Ct)}\cdot e.
\end{equation}
Applying the Stirling formula \eqref{eq:60a} in the above, one obtains
that 
\begin{equation}
  \label{eq:88}
      e^n(t) \leq \dfrac{\left(C\cdot e\cdot t\right)^n}{\sqrt{2\pi} n^{\frac{3}{2}}} +
      C\cdot t \cdot e \cdot e^{-n(1-Ct)}.
\end{equation}
We let 
\begin{equation*}
  t^* = \min\left\{ \dfrac{1}{2Ce},\,1,\, T\right\}.
\end{equation*}
Then, for all $0\leq t\leq t^*$,
\begin{equation}
  \label{eq:89}
      e^n(t) \leq \dfrac{1}{\sqrt{2\pi}
        n^{\frac{3}{2}}}\left(\dfrac{1}{2}\right)^n  +
      C\cdot T \cdot e \cdot e^{-\frac{1}{2}n}.
\end{equation}
Thus, $e^n(t)$ is a geometrically converging sequence independent of
the time $t$, and so is $\rho^n(t)$. We therefore have that
\begin{equation}
\label{eq:89a}
  \Phi^n_t(a) \longrightarrow \Phi_t(a)\qquad \textrm{ strongly in
  } L^\infty(0,t^*,\, L^1(\M)).
\end{equation}
We note that $t^*$ can be taken independently of $\epsilon$, and it
depends on $\M$, $|q_0|_\infty$, and $|f|_\infty$, and therefore the
above process can be repeated over intervals of length
$t^*$, until the whole interval $[0,\,T]$ is covered. The lemma is
thus proven.
\end{proof}

We now study the convergence of $q^n$. 
We define
\begin{equation}
  \label{eq:90}
  q(x, t) := q_0(\Phi_{-t}(x)) + \int_0^t
  f(\Phi_{s-t}(x),s)ds,\qquad \forall x\in\M.
\end{equation}
It is clear that $q\in L^\infty(Q_T)$, provided that $f\in
L^\infty(Q_T)$. We have the following result.
\begin{lemma}
Assume that $f\in L^\infty(0,T; C(\M))$. Then, for any $g\in C(\M)$, 
  \begin{equation}\label{eq:90a}
    \int_\M g(x) q^n(x,t)dx  \longrightarrow \int_\M g(x) q(x,t)dx
    \qquad \textrm{as } n\longrightarrow \infty.
  \end{equation}
The convergence is uniform in $t$.
\end{lemma}
\begin{proof}
We subtract \eqref{eq:90} from \eqref{eq:52b},
\begin{equation*}
q^{n+1}(x,t) -  q(x, t) = q_0(\Phi^n_{-t}(x)) -
q_0(\Phi_{-t}(x)) 
+ \int_0^t
 \left( f(\Phi^n_{s-t}(x),s) -  f(\Phi_{s-t}(x),s)\right) ds.
\end{equation*}
Multiplying the above equation with a $g\in C^1(\M)$, and integrate
over $\M$, we obtain
\begin{align*}
&\int_\M g(x) \left( q^{n+1}(x,t) -  q(x, t)\right) dx \\
= & \int_\M g(x) \left( q_0(\Phi^n_{-t}(x)) -
     q_0(\Phi_{-t}(x))\right) dx + \\
&\int_0^t\int_\M g(x)  \left( f(\Phi^n_{s-t}(x),s) -
     f(\Phi_{s-t}(x),s)\right) dx ds\\
= & \int_\M q_0(a) \left( g(\Phi^n_t(a)) - g(\Phi_t(a)\right) dx + \\
&\int_0^t\int_\M g(x)  \left( f(\Phi^n_{s-t}(x),s) -
     f(\Phi_{s-t}(x),s)\right) dx ds.
\end{align*}
For the moment, we assume that $g\in C^1(\M)$ and $f\in L^\infty(0,T;
C^1(\M))$. Then, we derive that
\begin{align*}
&\left| \int_\M g(x) \left( q^{n+1}(x,t) -  q(x, t)\right) dx\right| \\
\leq & |\nabla g|_\infty \cdot |q_0|_\infty \int_\M  \left| \Phi^n_{t}(a) -
     \Phi_{t}(a)\right| da + \\
& |g|_\infty \cdot |\nabla f|_\infty \int_0^t\int_\M  \left|\Phi^n_{s-t}(x) -
     \Phi_{s-t}(x)\right| dx ds.
\end{align*}
By the strong convergence of $\Phi^n_t$ in $L^\infty(0, T; L^1(\M))$,
the right-hand side above converges to zero as $n\longrightarrow
\infty$. Then, by a continuity argument, one can show that, for any
$g\in C(\M)$ and $f\in L^\infty(0,T; C(\M))$,
\begin{equation*}
  \int_\M g(x) (q^n(x,t) - q(x,t)) dx \longrightarrow
  0,\qquad\textrm{as } n\longrightarrow \infty.
\end{equation*}
\end{proof}

We now verify the convergence of the velocity field $\ub^n$.
Using the QGPV $q$, we define
\begin{equation}
\label{eq:u}
  \ub(x,t) = 
  \int_\M K(x),y)(q(y,t) -y_2)dy + l(q) \int_\M
  K(x),y)dy. 
\end{equation}

\begin{lemma}
  As $n\longrightarrow \infty$,
  \begin{equation*}
    \ub^n(x,t) \longrightarrow \ub(x,t)\qquad \textrm{strongly in }
    L^\infty(0,T; L^1(\M)).
  \end{equation*}
\end{lemma}
\begin{proof}
Subtracting \eqref{eq:u} from \eqref{eq:u-n}, we have
\begin{multline}
\label{eq:u-conv}
\ub^n(x,t) -  \ub(x,t) = 
  \int_\M K(x),y)(q^n(y,t) -q(y,t) )dy + \\
    (l(q^n) - l(q))  \int_\M K(x,y)dy. 
\end{multline}
From \eqref{eq:13d}, one obtains 
 \begin{equation*}
   l(q^n) - l(q) = -\dfrac{\int_\M \int_\M G(x,y)(q^n(y) - q(y)) dydx}{|\M| + \int_\M
     \int_\M G(x,y)dydx}.
 \end{equation*}
We note that $\int_M G(x,y) dx$ is continuous in $y$. Thus, by
\eqref{eq:90a}, we confirm that
\begin{equation*}
  l(q^n) - l(q) \longrightarrow 0 \qquad \textrm{as } n\longrightarrow
  \infty. 
\end{equation*}
We also note that $\int_\M K(x,y)dy$ is uniformly bounded. Therefore,
we have
\begin{equation}
  \label{eq:92}
  \left( l(q^n) - l(q)\right) \int_\M K(x,y)dy \longrightarrow 0
  \quad\textrm{uniformly in } x.
\end{equation}
For the first term on the right-hand side of \eqref{eq:u-conv}, we
substitute specifications  \eqref{eq:52b} and \eqref{eq:90} for $q^n$
and $q$, respectively, apply changes of variables, and we find that 
\begin{multline*}
  \int_\M K(x),y)(q^n(y,t) -q(y,t) )dy 
= \int_\M \left( K(x, \Phi^{n-1}_t(a)) - K(x,\Phi_t(a))\right) q_0(a)
   da +\\
 \int_0^t \int_\M \left( K(x, \Phi^{n-1}_{t-s}(b)) -
   K(x,\Phi_{t-s}(b))\right) f(b,s) db ds. 
\end{multline*}
Integrating the absolute value of the left-hand side on $\M$, and
using the relation above, we derive that 
\begin{align*}
  &\int_\M \left| \int_\M K(x),y)(q^n(y,t) -q(y,t) )dy\right| dx \\
\leq & |q_0|_\infty \int_\M \int_\M \left| K(x, \Phi^{n-1}_t(a)) -
       K(x,\Phi_t(a))\right| da dx  + \\
 &|f|_\infty \int_0^t \int_\M \int_\M \left| K(x, \Phi^{n-1}_{t-s}(b)) -
   K(x,\Phi_{t-s}(b))\right| dx\, db\, ds\\
\leq & C(|q_0|_\infty, |f|_\infty, \M) \left( \int_\M \chi(\delta^{n-1}_t(a)) da 
  +  \int_0^t \int_\M \chi(\delta^{n-1}_{t-s} (b)) db\, ds\right).
\end{align*}
The function $\delta^n_s$ is defined in \eqref{eq:72a}. Using the
bound \eqref{eq:57a} on $\chi$, we derive from the above that 
\begin{align*}
  &\int_\M \left| \int_\M K(x),y)(q^n(y,t) -q(y,t) )dy\right| dx \\
\leq & C(|q_0|_\infty, |f|_\infty, \M) \left( -\ln\epsilon \int_\M \delta^{n-1}_t(a) da 
  -\ln\epsilon \int_0^t \int_\M \delta^{n-1}_{t-s} (b)) db\, ds +
       \epsilon|\M| + \epsilon\cdot t|\M| \right).
\end{align*}
Thanks to  the uniform (in $t$) convergence \eqref{eq:89a} of
$\Phi^n_t$, it is clear that, for any $\epsilon >0$, there exists $N$
such that, for any $n>N$, 
\begin{equation}
\label{eq:93}
    \int_\M \left| \int_\M K(x),y)(q^n(y,t) -q(y,t) )dy\right| dx \leq
    C(\M, |q_0|_\infty, |f|_\infty, T) \cdot \epsilon.
\end{equation}
The constant $C$ is independent of $\epsilon$.
Combining \eqref{eq:92} and \eqref{eq:93}, we conclude that
\begin{equation}
  \label{eq:94}
  \ub^n(x,t) \longrightarrow \ub(x,t) \qquad\textrm{strongly in }
  L^\infty(0, T; L^1(\M)).
\end{equation}
\end{proof}


It remains to show that the limit $\Phi_t(a)$ of $\Phi^n_t(a)$ is
indeed the flow mapping for the limit velocity field $\ub(x,t)$ of
$\ub^n(x,t)$, i.e.~they satisfy the IVP \eqref{eq:51} in a certain
sense. 
\begin{lemma}
  For a.e.~$t\in (0,T)$, the flow mapping $\Phi_t(a)$ and the velocity
  field $\ub$ satisfy the  following relation,
\begin{equation}
  \label{eq:95}
    \Phi_t(a) = a + \int_0^t \ub(\Phi_s(a),s)ds \qquad \textrm{ in } L^1(\M).
\end{equation}
\end{lemma}
\begin{proof}
 Formally, equation \eqref{eq:95} is the limit of the integral
 equation \eqref{eq:70}. It has been shown that $\Phi^n_t$ converges
 to $\Phi_t$ in $L^1(\M)$, uniformly for $t\in (0, T)$. We only need
 to show that the integral on the right-hand side of \eqref{eq:70}
 converges to the integral on the right-hand side of \eqref{eq:95} in
 $L^1(\M)$ as well. To this end, we evaluate the $L^1$-norm of the
 difference of these two integrals, 
 \begin{multline*}
   \int_\M \left| \int_0^t \left(\ub^n(\Phi^n_s(a), s) -
   \ub(\Phi_s(a), s)\right) ds \right| da \leq \\
  \int_\M  \int_0^t \left|\ub^n(\Phi^n_s(a), s) -
   \ub^n(\Phi_s(a), s)\right| ds  da + \\
  \int_\M  \int_0^t \left|\ub^n(\Phi_s(a), s) -
   \ub(\Phi_s(a), s)\right| ds da.
 \end{multline*}
By Lemma \ref{lem:elliptic-reg}, each $\ub^n$ is quasi-Lipschitz
continuity, and the continuity parameter depends on $\M$,
$|q_0|_\infty$ and $|f|_\infty$ only, and not on $n$. Therefore, we
can write that 
 \begin{align*}
   &\int_\M \left| \int_0^t \left(\ub^n(\Phi^n_s(a), s) -
   \ub(\Phi_s(a), s)\right) ds \right| da \\
  \leq& C(\M, |q_0|_\infty) \int_0^t \int_\M\chi(|\Phi^n_s(a) -
   \Phi_s(a)| ds  da +
  \int_0^t \left|\ub^n(\cdot, s) -
   \ub(\cdot, s)\right|_{L^1(\M)}  ds\\
  \leq& C(\M, |q_0|_\infty) |\M| \int_0^t \chi\left(\dfrac{|\Phi^n_s -
   \Phi_s|_{L^1(\M)}}{|\M|} \right)ds  +
  \int_0^t \left|\ub^n(\cdot, s) -
   \ub(\cdot, s)\right|_{L^1(\M)}  ds
 \end{align*}
The convexity of the scalar function $\chi(\cdot)$ has been used in
deriving the last estimate. By the continuity of the function
$\chi(\cdot)$, the $L^1$ convergence of $\Phi^n$ and $\ub^n$, we
conclude that the above expression goes to zero as $n$ goes to
infinity, uniformly in $t\in (0, T)$. The claim is thus proven. 
\end{proof}


It is straightforward to verify that, if everything is smooth, then
the QGPV $q$ solves the
transport equation. Indeed, 
\begin{align*}
\dfrac{\p}{\p t} q + \ub\cdot\nabla q
&= \dfrac{\p}{\p t} q + \nabla q\cdot \dfrac{d}{dt}\Phi_t(a)\\
 &=  \dfrac{d}{d t} q(\Phi_t(a), t) \\
  &= f(\Phi_t(a), t).
\end{align*}
However, in general, $\Phi_t$ and $\ub$ do not satisfy \eqref{eq:51a}
in the classical 
sense, and the so-defined QGPV $q$ is not necessarily 
differentiable in time. We now show that  $q$ 
satisfies the transport equation in a weaker
sense. Indeed, using the change of variable $x = \Phi_t(a)$,
we verify that $q(x,t)$ satisfies the weak
formulation \eqref{eq:17}, 
\begin{align*}
  &\int_0^T\int_\M q(x,t)\left(\dfrac{\p\varphi}{\p t} +
    \ub\cdot\nabla\varphi\right) dxdt \\
  =& \int_0^T\int_\M \left(q_0(a) + \int_0^t
    f(\Phi_s(a),s)ds\right) \dfrac{d}{dt}\varphi(\Phi_t(a),t)dadt \\
  =& -\int_\M q_0(a)\varphi(a,0)da - \int_0^T\int_\M
  f(\Phi_t(a),t)\varphi(\Phi_t(a),t)dadt \\
  =& -\int_\M q_0(x)\varphi(x,0)dx - \int_0^T\int_\M
  f(x,t)\varphi(x,t) dx dt.
\end{align*}
The process goes through thanks to the fact that the mapping $x =
\Phi_t(a)$ preserves area. Thus we have proven the following result.

\begin{theorem}\label{thm-existence}
Assume that $f\in L^\infty(0,T; C(\M))$. Then there exists a solution
to the weak problem \eqref{problem} in 
  $L^\infty(0,T; V)$. 
\end{theorem}

In the previous section, we have shown that the initial and boundary
conditions are only satisfied in a weak sense. We will now show that
the solution $\psi$ actually enjoys better regularity, and satisfies
the initial and boundary conditions in the classical sense. 
\begin{theorem}\label{thm:reg}
  The initial and boundary conditions \eqref{eq:6}--\eqref{eq:8} are
  satisfied  in the classical sense, and $\Delta\psi$, $\p^2\psi/\p
  x\p t$ are strongly continuous with respect to $t$ on $[0,T]$ in
  $\lpm$ for any $p>1$. 
\end{theorem}
\begin{proof}
  We let $\varphi\in \mathcal{C}^\infty_c(\M)$. We multiply
  \eqref{eq:4} by $\varphi$ and integrate over
  $\M\times[\tau_1,\,\tau_2]$ for some $0\leq \tau_1<\tau_2\leq T$,
  \begin{multline*}
    (\Delta\psi - \psi,\,\varphi)|_{t=\tau_2} - (\Delta\psi -
    \psi,\,\varphi)|_{t = \tau_1} - \int_{\tau_1}^{\tau_2} \int_\M
    (\Delta\psi + y - \psi)\nabla^\perp\psi \cdot \nabla\varphi dx dt
    \\ =
    \int_{\tau_1}^{\tau_2} \int_\M f(x,t)\varphi(x) dx,
  \end{multline*}
  \begin{equation*}
    (q(\cdot,\tau_2),\,\varphi) - (q(\cdot,\tau_1),\,\varphi) =
    \int_{\tau_1}^{\tau_2} \int_\M 
     q \nabla^\perp\psi \cdot \nabla\varphi + f\varphi dx dt.
  \end{equation*}
We note that $\psi\in L^\infty(0,T; \,V)$, $q$ is bounded in
$L^\infty(Q_T)$, and $\nabla^\perp\psi$ is uniformly bounded in
$Q_T$. Thus, as $\tau_2 \longrightarrow \tau_1$,
\begin{equation}
  \label{eq:64}
  q(\cdot,\tau_2) \rightharpoonup q(\cdot,\tau_1)\qquad\textrm{in any }
  L^p(\M). 
\end{equation}
Writing \eqref{eq:90} over the interval $[\tau_1, \tau_2]$, one can
easily derive that, for $\forall\, p>1$, 
\begin{equation*}
  |q(\cdot,\tau_2)|_{\lpm}\leq |q(\cdot,\tau_1)|_\lpm +
  \int_{\tau_1}^{\tau_2} |f(\cdot,t)|_\lpm dt.
\end{equation*}
From this estimate we conclude that 
\begin{equation*}
  \overline\lim_{\tau_2\rightarrow\tau_1} |q(\cdot,\tau_2)|_\lpm \leq
  |q(\cdot,\tau_1)|_\lpm. 
\end{equation*}
In view of this estimate and the weak convergence \eqref{eq:64}, the
Radon-Riesz theorem applies, and we have
\begin{equation}
  \label{eq:q-cont}
  q(\cdot, t) \in C([0,T], L^p(\M)),\qquad \forall p>1.
\end{equation}

Concerning the continuity of $\p^2/\p x\p t$, we rewrite \eqref{eq:4}
as 
\begin{equation*}
    (\Delta - I) \dfrac{\p}{\p t}\psi =  \nabla\times F -
    \nabla\cdot\left(\nabla^\perp\psi(\Delta\psi + y - \psi)\right). 
\end{equation*}
Then,
\begin{equation*}
    \nabla \dfrac{\p}{\p t}\psi =  \nabla(\Delta-I)^{-1}\nabla\times F -
    \nabla(\Delta -
    I)^{-1}\nabla\cdot\left(\nabla^\perp\psi(\Delta\psi + y -
      \psi)\right),  
\end{equation*}
where $(\Delta - I)^{-1}$ is the solution operator of the elliptic
boundary value problem \eqref{eq:9}. We note that $q = \Delta\psi + y
- \psi$ is continuous in $t$ in any $\lpm$ with $p>1$, and
$\nabla^\perp\psi$ is uniformly bounded in $Q_T$. Thus, thanks to the
continuity of the differential operator $\nabla(\Delta -
I)^{-1}\nabla(\cdot)$, and provided that $F$ is continuous in $t$ as
well, 
\begin{equation}
  \label{eq:psi_x_t}
\nabla\frac{\p\psi}{\p t}(\cdot, t) \in C([0,T];
L^p(\M)),\qquad\forall\, p>1. 
\end{equation}

By Lemma \ref{lem:elliptic-reg},
\begin{equation*}
  |\psi(\cdot,\tau_2) - \psi(\cdot,\tau_1)|_{W^{2,p}(\M)} \leq Cp
  |q(\cdot,\tau_2) - q(\cdot,\tau_1) |_\lpm. 
\end{equation*}
Thus, as $\tau_2\longrightarrow \tau_1$, 
\begin{equation*}
  \psi(\cdot,\tau_2) \longrightarrow \psi(\cdot,\tau_1)\qquad
  \textrm{in } W^{2,p}(\M),
\end{equation*}
Thus the initial condition \eqref{eq:8} is satisfied in a stronger norm,
\begin{equation*}
  \psi(\cdot,0) = \psi_0(\cdot)\qquad\textrm{ in } W^{2,p}(\M). 
\end{equation*}

We also  note that $\psi\in L^\infty(0,T; V)$ implies that
\begin{equation}
  \label{eq:67}
  \dfrac{\p\psi}{\p x} \in L^\infty(0,T; W^{1,p}(\M)).
\end{equation}
From Lemma \ref{lem-t-deriv}, we have
\begin{equation}
  \label{eq:68}
  \dfrac{\p^2\psi}{\p t\p x} \in L^\infty(0,T; \lpm) \subset L^p(Q_T).
\end{equation}
Combining \eqref{eq:67} and \eqref{eq:68}, we derive that 
\begin{equation}
  \label{eq:69}
  \dfrac{\p\psi}{\p x} \in W^{1,p}(Q_T),\qquad\forall\, p> 1. 
\end{equation}
We take a $p > 3$. Then, by the Sobolev imbedding theorem,
\begin{equation}\label{eq:69a}
  \dfrac{\p\psi}{\p x} \in C^{0,\lambda} (Q_T)\qquad \textrm{for some
  } 0 < \lambda < 1.
\end{equation}
Thus, the streamfunction $\psi$ is continuous in the spatial-temporal
domain, and the initial and boundary conditions are satisfied in the
classical sense. 

Finally, \eqref{eq:q-cont} combined with \eqref{eq:69a} implies that 
\begin{equation}
  \label{eq:delpsi-cont}
  \Delta \psi \in C([0,T]; L^p(\M)).
\end{equation}
\end{proof}

We point out that, thanks to \eqref{eq:q-cont}, the QGPV $q$ assumes
its initial value $q_0$ in the $L^p$-norm, for any $p>1$, which is an
improvement over \eqref{wc-3}. 

\section{Concluding remarks}
So far, theoretical studies of geophysical models have largely focused
on those with ``rigid'' lids on the top. Models with free or
deformable top surfaces are much harder, and belong to the class of
free boundary problems. Results on such problems are still scarce. The
few published results concern the local existence and uniqueness for
the viscous PEs with a free top surface (\cite{Honda2012-zh,
  Honda2015-rc}). This is not surprising. When a top surface is left
free, it may break, as it does in reality. The current work deals with
the inviscid barotropic QG equation. The free top surface enters the
dynamics through its effect on the QGPV, thanks to Kelvin's
Circulation Theorem. The current work confirms that, when the free
surface is included in this way, the QG model remains globally
well-posed. 

The current work is part of a project to address the well-posedness of
invicid QG equations. The other QG models that are being or will be
considered include the multi-layer QG model and the three-dimensional
QG model. Within the multi-layer QG model, besides the top surface,
the interior layer interfaces are also free to deform. Physically, the
deformation of the interior interfaces can be much more significant
than those of the top surface, thanks to the reduced gravity in the
interior of the fluid (\cite{Pedlosky1987-gk}). A mathematical
challenge posed by the multi-layer QG model is the fact that the
linear differential operator in the QGPV specification is not negative
definite, a departure from the barotropic case. 
The three-dimensional QG equations are posed on a truly
three-dimensional domain, but its velocity field remains
horizontal. Hence this system is much more complex than the barotropic
or multi-layer QG equations, but are notably simpler than
 fully three-dimensional fluid models, including the
Navier-Stokes/Euler equations, and the primitive equatons. 
These problems will be addressed in forthcoming papers. 



\appendix
\section{Proof of Lemma \ref{lem:flow}}\label{ap-1}
The proof is by the Picard iteration technique. For the moment, we
drop the super index from \eqref{eq:51}, and write
the initial boundary
value problem in the integral form,
\begin{equation}
  \label{eq:54}
  \Phi_t(a) = a + \int_0^t \ub(\Phi_s(a),s)ds.
\end{equation}
We let 
\begin{equation}
  \label{eq:55}
  \Phi^0_t(a) = a, \qquad t\ge 0.
\end{equation}
Assuming that $\Phi^{k-1}_t(a)$ is known, we compute $\Phi^k_t$ by
\begin{equation}
  \label{eq:56}
\Phi^k_t(a) = a + \int_0^t \ub(\Phi^{k-1}_s(a),s)ds.  
\end{equation}
When $n=1$, using \eqref{eq:55} and \eqref{eq:56}, and the uniform
boundedness of $\ub$, we obtain
\begin{equation}
  \label{eq:57}
  |\Phi^1_t(a) - \Phi^0_t(a)| \leq \int_0^t |
  \ub(\Phi^{k-1}_s(a),s)|ds \leq Ct. 
\end{equation}
We estimate the difference $\Phi^k_t(a) - \Phi^{k-1}_t(a)$ using
\eqref{eq:56} and the quasi-Lipschitz condition on $\ub$,
\begin{align*}
  |\Phi^k_t(a) - \Phi^{k-1}_t(a)| 
   &\leq \int_0^t |\ub(\Phi^{k-1}_s(a),s) -
     \ub(\Phi^{k-2}_s(a),s)| ds\\
  & \leq C\int_0^t \chi\left(|\Phi^{k-1}_s(a) - \Phi^{k-2}_s(a)|\right) ds.
\end{align*}
Here, $C$ is, by assumption, independent of $a$ or $t$, and the
scalar function $\chi$ is defined in Lemma \ref{lem:elliptic-reg}. We
note that, for an arbitrary $0 < \epsilon < 1$, this scalar function
is bounded by 
\begin{equation}
\label{eq:57a}
  \chi(r) \leq -\ln\epsilon\cdot r + \epsilon,\qquad\forall\, r\geq 0.
\end{equation}
Hence, we have
\begin{equation}
  \label{eq:58}
  |\Phi^k_t(a) - \Phi^{k-1}_t(a)| \leq -C\ln\epsilon \int_0^t
  |\Phi^{k-1}_s(a) - \Phi^{k-2}_s(a)|ds + C\epsilon t.  
\end{equation}
Using \eqref{eq:57}, we find that
\begin{equation*}
    |\Phi^2_t(a) - \Phi^1_t(a)| \leq -\ln\epsilon\cdot
    C^2\dfrac{t^2}{2} + C\epsilon t.
\end{equation*}
Then, by induction, we find that
\begin{equation}
  \label{eq:59}
    |\Phi^k_t(a) - \Phi^{k-1}_t(a)| \leq C^k
    \left(-\ln\epsilon\right)^{k-1} \dfrac{t^k}{n!} +
    \epsilon\sum_{k=1}^{k-1} C^k\left(-\ln\epsilon\right)^{k-1}
      \dfrac{t^k}{k!}. 
\end{equation}
The summation on the right-hand side can be bounded by an exponential
function,
\begin{equation*}
      \epsilon\sum_{k=1}^{k-1} C^k\left(-\ln\epsilon\right)^{k-1}
      \dfrac{t^k}{k!} \leq C\epsilon t e^{-C\ln\epsilon t} = Ct\epsilon^{1-Ct}.
\end{equation*}
Thus,
\begin{equation}
  \label{eq:60}
    |\Phi^k_t(a) - \Phi^{k-1}_t(a)| \leq
    Ct\dfrac{(-C\ln\epsilon\cdot t)^{k-1}}{n!} + Ct\epsilon^{1-Ct}. 
\end{equation}
This estimate holds for arbitrary $0 < \epsilon < 1$. We take
$\epsilon = e^{-k}$. We derive from \eqref{eq:60} that
\begin{equation*}
      |\Phi^k_t(a) - \Phi^{k-1}_t(a)| \leq C^k\cdot
      n^{k-1}\dfrac{t^k}{n!} + Ct e^{-k(1-Ct)}. 
\end{equation*}
By the Stirling formula,
\begin{equation}
\label{eq:60a}
  \dfrac{n^{k-1}}{k!} \leq \dfrac{e^k}{\sqrt{2\pi}k^{\frac{3}{2}}}. 
\end{equation}
Thus, 
\begin{equation*}
      |\Phi^k_t(a) - \Phi^{k-1}_t(a)| \leq 
      \dfrac{(C e t)^k}{\sqrt{2\pi}k^{\frac{3}{2}}}  + Ct e^{-k(1-Ct)}. 
\end{equation*}
We choose
\begin{equation*}
  t^\ast = \dfrac{1}{2Ce}.
\end{equation*}
Then, 
\begin{equation*}
  Cet^\ast = \dfrac{1}{2} \quad \textrm{and}\quad 1-Ct^\ast \geq
  \dfrac{1}{2}, 
\end{equation*}
and, for all $0\leq t\leq t^\ast$,
\begin{equation*}
      |\Phi^k_t(a) - \Phi^{k-1}_t(a)| \leq 
 \dfrac{1}{\sqrt{2\pi}k^{\frac{3}{2}}}\left(\dfrac{1}{2}\right)^k  +
 Ct^\ast e^{-\frac{1}{2}k}.  
\end{equation*}
Thus, for $0\leq t\leq t^\ast$, $|\Phi^k_t(a) - \Phi^{k-1}_t(a)|$
is a convergent  geometric sequence, and therefore $\Phi^k_t(a)$ is
Cauchy and converges to a limit function $\Phi_t(a)$, which solves
the integral equation \eqref{eq:54} and thus the IVP \eqref{eq:51} on
the interval $[0,\,t^\ast]$. The choice of $t^\ast$ depends on the
generic constant $C$ only, and is independent of the initial position
$a$. Hence, the same procedure can be applied to extend the solution
$\Phi_t(a)$, for every $a\in\M$, to the whole time interval
$[0,\,T]$. 

For uniqueness, we assume that $\Psi_t(a)$ is another solution
satisfying 
\begin{equation}
  \label{eq:61}
  \Psi_t(a) = a + \int_0^t \ub(\Psi_s(a),s)ds,
\end{equation}
and $\Psi_t(a)$ differs from $\Phi_t(a)$ for every
$t\in(0,\,\sigma)$ for some $\sigma>0$. If this is not the case, we
can always move the initial point to where $\Phi_t(a)$ and
$\Psi_t(a)$ start to fork. Subtracting this equation from
\eqref{eq:54}, and using the quasi-Lipschitz condition on $\ub$, we
find
\begin{equation*}
  |\Phi_t(a) - \Psi_t(a)| \leq -C \ln\epsilon \int_0^t
  |\Phi_s(a) - \Psi_s(a)|ds + C\epsilon t, \quad
  \forall\,0<\epsilon < 1. 
\end{equation*}
One can obtain an estimate on the difference using the Gronwall
inequality, 
\begin{equation*}
  |\Phi_t(a) - \Psi_t(a)| \leq \dfrac{\epsilon}{-\ln\epsilon}
  \left( e^{-Ct\ln\epsilon} - 1\right).
\end{equation*}
This inequality holds for arbitrary $0<\epsilon<1$, and for all $t\in
[0,T]$. We take $\epsilon = e^{-k}$ for some integer $n>0$. Then
\begin{equation*}
  |\Phi_t(a) - \Psi_t(a)|\leq \dfrac{e^{-k}}{n}\left(e^{Ctn} -
    1\right) = \dfrac{1}{n} e^{(Ct-1)n} - \dfrac{e^{-k}}{n}. 
\end{equation*}
We set $t^\ast = \dfrac{1}{2C}$. Then
\begin{equation*}
  Ct^\ast - 1 = -\dfrac{1}{2}
\end{equation*}
and for $0\leq t\leq t^\ast$,
\begin{equation*}
  |\Phi_t(a) - \Psi_t(a)|\leq  \dfrac{1}{n} e^{-\frac{1}{2}k} -
  \dfrac{e^{-k}}{k}.  
\end{equation*}
We note that this estimate holds for arbitrary $k$'s. For this to be
possible, $\Phi_t(a)$ and $\Psi_t(a)$ must agree on
$[0,\,t^\ast]$, which contradicts the assumption on $\Psi_t$. Hence
$\Phi_t(a)$ is a unique solution.

\bibliographystyle{siamplain}
\bibliography{references}
\end{document}